\documentclass{amsart}
\usepackage{amssymb,amsmath,amsthm,amsfonts,amsopn,url,bbm, hyperref}
\usepackage{amscd, enumitem}
\usepackage{color}
\theoremstyle{plain}
\newtheorem{thm}{Theorem}[section]
\newtheorem{prop}[thm]{Proposition}
\newtheorem{lemma}[thm]{Lemma}
\newtheorem{keylem}[thm]{Key Lemma}
\newtheorem{cor}[thm]{Corollary}
\theoremstyle{definition}
\newtheorem{defn}[thm]{Definition}
\newtheorem*{defn*}{Definition}
\newtheorem{example}[thm]{Example}
\newtheorem*{rmk*}{Remark}
\newtheorem{disc}[thm]{Discussion}
\newtheorem*{disc*}{Discussion}
\newcommand{\arrow}[1]{\stackrel{#1}{\rightarrow}}
\newcommand{\vect}[2]{{{#1}_1, \dotsc, {#1}_{#2}}}
\newcommand{\wh}[1]{\widehat{#1}}
\newcommand{\inc}{\subseteq}
\newcommand{\inj}{\hookrightarrow}
\newcommand{\C}{\mathbb{C}}
\newcommand{\N}{\mathbb{N}}
\newcommand{\R}{\mathbb{R}}
\newcommand{\Z}{\mathbb{Z}}
\newcommand{\Q}{\mathbb{Q}}

\newcommand{\blank}{\,\underline{\hphantom{J}}\,}

\newcommand{\cA}{\mathcal{A}}
\newcommand{\cC}{\mathcal{C}}
\newcommand{\cF}{\mathcal{F}}

\newcommand{\cJ}{\mathcal{J}}
\newcommand{\cK}{\mathcal{K}}
\newcommand{\cL}{\mathcal{L}}
\newcommand{\cM}{\mathcal{M}}
\newcommand{\cN}{\mathcal{N}}
\newcommand{\cP}{\mathcal{P}}

\newcommand{\cS}{\mathcal{S}}
\newcommand{\cT}{\mathcal{T}}
\newcommand{\cV}{\mathcal{V}}
\newcommand{\cW}{\mathcal{W}}
\newcommand{\fA}{\mathfrak{A}}
\newcommand{\fB}{\mathfrak{B}}

\newcommand{\m}{\mathfrak{m}}
\newcommand{\n}{\mathfrak{n}}
\newcommand{\p}{\mathfrak{p}}
\newcommand{\fq}{\mathfrak{q}}

\newcommand{\imp}{\implies}
\newcommand{\dbimp}{\iff}

\newcommand{\ov}[1]{\overline{#1}}

\newcommand{\icr}[1]{#1'}

\newcommand{\ici}[1]{#1^-}
\newcommand{\icip}[1]{(#1)^-}

\newcommand{\oI}{\ici I}
\newcommand{\surj}{\twoheadrightarrow}

\newcommand{\red}{_{\mathrm{red}}}

\newcommand{\h}{^{\mathrm{h}}}

\newcommand{\ct}{^{\mathrm{cont}}}
\newcommand{\ctx}{^{\mathrm{cont},x}}
\newcommand{\ctz}{^{\mathrm{cont},(0,0)}} 
\newcommand{\ctv}{^{\mathrm{cont},v}}

\newcommand{\ctzz}{^{\mathrm{cont},(0,\dotsc,0)}}
\newcommand{\ax}{^{\mathrm{ax}}}
\newcommand{\AX}{^{\mathrm{AX}}}
\newcommand{\rmk}{(R,\,m,\,K)}

\newcommand{\z}{\mathbf{0}}
\newcommand{\spi}[2]{{#1}^{-\!#2 \mathrm{sp}}}
\newcommand{\inn}[1]{{#1}_{>1}}
\newcommand{\ncl}{^{\natural}}
\newcommand{\nmx}{^{\natural^{\mathrm{mix}}}}

\DeclareMathOperator{\di}{dim}
\DeclareMathOperator{\Spec}{Spec}

\DeclareMathOperator{\Ker}{Ker}

\DeclareMathOperator{\Jac}{Jac}
\DeclareMathOperator{\MaxSpec}{Max\, Spec}
\DeclareMathOperator{\Rad}{Rad}
\DeclareMathOperator{\ord}{ord}
\DeclareMathOperator{\gr}{gr}
\DeclareMathOperator{\depth}{depth}

\DeclareMathOperator{\md}{mod}
\newcommand{\rv}{relevant}
\newcommand{\mnb}{MN-basic}
\newcommand{\mnc}{MN-closed}

\author{Neil Epstein}
\address{Department of Mathematical Sciences \\ George Mason University \\ Fairfax, VA  22030}
\email{nepstei2@gmu.edu}
 
\author{Melvin Hochster}
\address{Department of Mathematics\\
University of Michigan\\
Ann Arbor, MI 48109-1043\\
USA}
\email{hochster@umich.edu}
\title{Continuous closure, axes closure, and natural closure}

\date{July 2, 2015}
\begin{document}
\begin{abstract}
Let $R$ be a reduced affine $\C$-algebra, with corresponding affine algebraic set $X$.  Let $\cC(X)$ be the ring of continuous (Euclidean topology)
$\C$-valued functions on $X$.  Brenner defined the \emph{continuous closure} $I\ct$ of an ideal $I$ as $I\cC(X) \cap R$.  He also introduced an algebraic notion of \emph{axes closure} $I\ax$  that always contains $I\ct$, and asked whether they coincide.  We extend
the notion of axes closure to general Noetherian rings, defining $f \in I\ax$ if its image is in $IS$ for every homomorphism $R \to S$, where  $S$ is a one-dimensional complete seminormal local ring.  We also introduce the \emph{natural closure} $I\ncl$ of $I$. One of many characterizations is  
$I\ncl = I + \{f \in R: \exists n >0 \mathrm{\ with\ } f^n \in I^{n+1}\}$.   We show that $I\ncl \subseteq I\ax$, and that when continuous closure is defined,  
$I\ncl \subseteq I\ct \subseteq I\ax$.  Under mild hypotheses on the ring, we show that $I\ncl = I\ax$ when $I$ is primary to a maximal ideal, and that if $I$ has no embedded primes, then $I = I\ncl$ if and only if $I = I\ax$, so that $I\ct$  agrees as well. We deduce that in the polynomial ring
$\C[\vect x n]$, if $f = 0$ at all points where all of the ${\partial f \over \partial x_i}$ are 0, then
$f \in ( {\partial f \over \partial x_1}, \, \ldots, \, {\partial f \over \partial x_n})R$.  We characterize $I\ct$ for 
monomial ideals in polynomial rings over $\C$, but we show that the inequalities $I\ncl \subset I\ct$ and $I\ct \subset I\ax$ can be strict 
for monomial ideals even in dimension 3.  Thus, $I\ct$ and $I\ax$ need not agree, although we prove they are equal in $\C[x_1, x_2]$.
\end{abstract}
 \thanks{The second named author is grateful for support from the National Science Foundation, grants DMS-0901145 and DMS-1401384.}
\subjclass[2010]{Primary 13B22, 13F45; Secondary 13A18, 46E25, 13B40, 13A15}
\keywords{continuous closure, axes closure, natural closure, seminormal ring}

\maketitle
\setcounter{tocdepth}{1}
\tableofcontents

\section{Introduction}\label{sec:intro}

Holger Brenner \cite{Br-cc} introduced a new closure operation on
ideals in finitely generated $\C$-algebras called {\it continuous closure},
and asks whether it is the same as an algebraic notion called {\it axes
closure} that he introduces.  He proves this for ideals in a polynomial
ring that are primary to a maximal ideal and generated by monomials.
We shall relate this closure to some variant notions of integral closure,
{\it special part of the integral closure} over a local ring, introduced in \cite{nme-sp}, and
{\it inner integral closure}, a notion explored here that exists without an explicit name
in the literature, and also to a notion we introduce called {\it natural closure}.
We shall prove that if $I$ is an unmixed ideal in any affine $\C$-algebra,
then $I$ is continuously closed if and only if it is axes closed.  See Theorem~\ref{thm:prnatax},
Corollary~\ref{cor:prctax}, and Corollary~\ref{cor:zerodimctax}.  We also provide further
conditions under which continuous closure equals axes closure or natural closure.

In consequence we can prove, for example, that if $f$ is a polynomial in $\C[\vect x n]$
that vanishes wherever its partial derivatives all vanish, then there
are continuous functions $g_j$ from $\C^n \to \C$ such that
\[
f = \sum_{j=1}^n   g_j \frac{\partial f}{\partial x_j}.
\]
See Theorem~\ref{thm:Jac}.

On the other hand, we show that continuous closure is sometimes
strictly smaller than axes closure.  Indeed, in \S\ref{sec:counter}, we 
give an example (followed by a method of generating such examples)
 of a monomial ideal in a polynomial ring over $\C$ which is 
continuously closed but not axes closed. 

After hearing the second named author give a talk on the results of this paper, Koll\'ar \cite{Kol-cc} 
studied continuous closure in the context of coherent sheaves on schemes over $\C$ and
has given an algebraic characterization that permits the notion of continuous closure
to be defined in a larger context.  In a further paper  \cite{FK-cc}, continuous closure is studied
over topological fields other than $\C$, particularly for the field of real numbers. 

Let $R$ be a finitely generated $\C$-algebra.  Map a polynomial ring
$\C[\vect x n] \surj R$ onto $R$ as $\C$-algebras.  Let $\fA \inc \C[\vect x n]$
be the kernel ideal, and let $X$ be the set of points in $\C^n$ where all
elements of $\fA$ vanish.  $X$ may be identified with the set of maximal
ideals of $R$.  Then $X$ has a Euclidean topology, and
the topological space $X$  is independent of the presentation of $R$.   
We let  $\cC(Y)$ denote the ring of complex-valued continuous functions
on any space $Y$.  Polynomial functions on $\C^n$,  when restricted to $X$,
yield a ring $\C[X]$ which is isomorphic to $R\red$,  the original ring modulo
the ideal $N$ of nilpotents.  (Nothing will be lost in the sequel if we restrict
attention to reduced rings $R$ (\emph{i.e.}, rings without nonzero nilpotents).)  Thus,
we have a  $\C$-homomorphism  $R \to \cC(X)$ which is injective when
$R$ is reduced.  The {\it continuous closure} if $I \inc R$, denoted $I\ct$, is the contraction
of $I\cC(X)$ to $R$.  That is,   if  $I = (\vect f m)R$,  then $f \in I\ct$ precisely
when there are continuous functions  $g_i:X \to \C$ such that
\[
f|_X = g_1f_1|_X + \cdots + g_mf_m|_X,
\]
where   $h|_X$  indicates
the image of $h \in R$ in $\cC(X)$.  Henceforth, we focus on the case
where $R$ is reduced, and omit $\quad|_X$  from the notation.  However,
we can state many of the results without this hypothesis:  one can typically
pass at once in the proofs to the case where the ring is reduced.

In this paper we study this closure and several other closures that are related, obtaining satisfying answers to 
many quesstions that were open even for polynomial rings.

Let $L$ be an algebraically closed field.  We are especially concerned with the case where
$L = \C$ is the complex numbers.
A finitely generated $L$-algebra $R$ is called a {\it ring of axes} over $L$ if it is
one-dimensional reduced and either smooth, with just one irreducible component,
or else is such that the corresponding algebaic set is the union of $n$ smooth irreducible curves,  
and there is a unique singular point, which is the intersection of any two
of the components, such that the completion of the local ring at that point  is isomorphic with
$L[\![\vect x n]\!]/(x_ix_j \mid 1 \leq i < j \leq n)$.  

 We now restrict to the case of the complex numbers. In \cite{Br-cc} Brenner obtains a structure
theorem for the ideals of such a completed local ring that enables him to prove
that in a ring of axes over $\C$, for every ideal $I$,  $I = I\ct$.  The {\it axes} closure $I\ax$ of an
ideal $I$ of $R$ is defined to be the set of elements $r$ such that for every
$\C$-homomorphism $R \to S$, where $S$ is a ring of axes,  one has  $r \in IS$.
The results of \cite{Br-cc} imply that $I\ct \inc I\ax$ in general, and  that they
agree for ideals of polynomial rings that are primary to maximal ideals and 
are generated by monomials.   As mentioned above, we prove here that
continuous closure coincides with axes closure for all  ideals of
affine $\C$-algebras that are primary to a maximal ideal, and in many other cases.
We also show that an unmixed ideal (one that has no embedded primes) is axes closed if and only if it is 
continuously closed, and that there exist continuously closed ideals which
are not axes closed, which answers a question raised by Brenner.  

In \S3 we prove that an element 
$r$ of an affine $\C$-algebra is in the axes closure of $I \inc R$ if and only
if $x \in IS$ for every homomorphism of $R$ to an excellent (respectively,
complete) Noetherian one-dimensional seminormal ring $S$.  We use the
latter definition to extend the notion of axes closure to all Noetherian rings.  
See Corollaries~\ref{cor:axsemiC} and \ref{cor:axsemi} and Definition~\ref{def:axsemi}.

Here is a brief sketch of the contents of the paper:

In \S\ref{sec:contcl}, we  discuss some important properties of continuous closure that we will need. Some of this material 
 is reviewed from \cite{Br-cc}, but in  some cases we need
sharper or more general results.   \S\ref{sec:seminor} is devoted to seminormal rings and their connections to continuous and axes closures.  In \S\ref{sec:axes}, we extend the definition of axes closure to general Noetherian rings, characterizing it by maps to excellent one-dimensional seminormal rings, and we show that this 
agrees with the original definition in Brenner's setting.  In \S\ref{sec:nat} we discuss the concepts of special and inner integral closure, and introduce the notion of natural closure. We also introduce the notion of $I$-relevant ideals, which are used to
characterize when an ideal is naturally closed, and which play
a key role in proving the results of \S\ref{sec:natax}. We show that the natural closure is contained in the axes closure and, 
wherever it is defined, the continuous closure. This ``traps'' continuous closure between two algebraically defined closures.  This is the main tool
used in \S~\ref{sec:natax} to prove results on when axes closure and continuous closure agree.

One of the main results of \S\ref{sec:partial} has already been stated in the second paragraph of 
this Introduction.  

\S\ref{sec:natax} is mostly devoted to a number of important cases where natural closure and axes closure agree, and
contains several of our main results.
When these two agree and continuous closure is defined, it agrees as well.  This yields the central result than an unmixed ideal in
an affine $\C$-algebra is continuously closed if and only if it is axes closed.  We also give a characterization of seminormal rings in terms of axes closed ideals.  
 
 In \S\ref{sec:dimtwo} we show that continuous and axes closure agree in the locally factorial two-dimensional case.  In 
 \S\ref{sec:counter}, we develop a ``fiber criterion'' to exclude certain elements from the continuous closure of an ideal. 
This allows us to construct examples of continuously closed ideals that are not axes closed (even a monomial ideal in a three-dimensional polynomial ring). We apply this criterion in \S\ref{sec:monomial} to show that for monomial ideals in polynomial rings over $\C$, continuous closure always equals mixed natural closure, which is defined in that section.  Finally, we introduce in \S\ref{sec:bigax} a closure operation $\AX$ that is similar to $\ax$, and agrees with it in equal characteristic $0$, but is based on weakly normal rings instead of seminormal ones, and thus is sometimes bigger in positive characteristic. The two notions and their relative usefulness are discussed.

We conclude this introduction by reminding the reader of the 
definition of a term we have already used several times: 
\begin{defn}\label{def:closure}
A \emph{closure operation} $^{\#}$ on (the ideals of) a ring $R$ is an inclusion preserving function from ideals to ideals
 such that if the value on $I$ is denoted $I^{\#}$, then for all ideals $I \inc R$,  $I \inc I^\# = (I^\#)^\#$.
\end{defn}

We refer the reader to \cite{nme-guide2} for a detailed treatment of closure operations and their properties.

\section{Properties of continuous closure}\label{sec:contcl}
Given a homomorphism $R \to S$ of finitely generated $\C$-algebras
we get an induced map in the other direction of the corresponding
algebraic sets, $X \leftarrow Y$,  which is continuous in the
Euclidean topologies (it is defined coordinatewise by restricted
polynomial functions),  and so there is a commutive diagram
\[\begin{CD}  \cC(R) @>>> \cC(S)\\
             @AAA           @AAA\\
             \C[X]  @>>> \C[Y]\\
             @AAA           @AAA\\
              R  @>>> S\end{CD}\]
where the top and middle horizontal arrows are induced by the
map $Y \to X$. 

Hence (cf.\ \cite{Br-cc}):
\begin{prop}[persistence of continuous closure]\label{pr:pers}
If
$h: R \to S$ is a homomorphism of finitely generated
$\C$-algebras, $I$ is an ideal of $R$, and $f \in I\ct$,  then
$h(f) \in (IS)\ct$. 
\end{prop}

If  $I$ is an ideal of a ring $R$ and  $\cF$ is a subset of $R$,  let
$I:_R \cF = \{r \in R \mid \hbox{for all\ }f \in \cF, fr \in I\}$.  If  $\cF$ consists
of a single element $f$,  this coincides with $I:_R f = I:_R fR$.  
Note that $I:_R \cF = \bigcap_{f \in \cF} (I:_R f)$,  and that if $J$ is the
ideal generated by $\cF$ then $I:_R \cF = I:_RJ$.

\begin{prop}
Let $I$ be an ideal of an affine $\C$-algebra
$R$, and $\cF \inc R$.  If $I$ is a continuously closed ideal, so is
$I:_R \cF$  for every set $\cF \inc R$,  and so is the contraction of $IR_W$
to $R$ for every multiplicative system $W$. 
\end{prop}

\begin{proof}
The second statement follows from the first, because the contraction
of $IR_W$ to $R$ is the union of the ideals $I :_R w$ for $w \in W$,
and since this set is directed, one can choose $w \in W$ so that the contraction
is the same as $I :_R  w$.  Moreover, the statement for $\cF$ reduces to the
case of a single element $f$, since an intersection of continuously closed ideals
is evidently continuously closed.

Now suppose that $r \in R$ is a linear combination with continuous coefficients
$\vect g h$ of elements $\vect f h$ of $I:_R f$.  Then $fr = \sum_{i=1}^h g_h(ff_i)$
where every $ff_i \in I$,  and so $fr \in I\ct = I$,  and $r \in I:_Rf$,  as required. 
\end{proof}

\begin{cor}\label{cor:ctcomponents}
If $R$ is an affine $\C$-algebra and $I$ is a continuously
closed ideal of $R$, then so is every primary component of $I$ for a minimal
prime $P$ of $I$.
\end{cor}
\begin{proof}  The minimal primary component corresponding to $P$ is the contraction
of $IR_W$ to $R$,  with $W = R-P$.
\end{proof}

For any ring homomorphism  $R \to S$,  if $I,\,J \inc R$ the product of the contractions of $IS$ and $JS$
is obviously contained in the contraction of $(IJ)S = (IS)(JS)$.  Applying this to the map $\C[X] \to \cC(X)$,
we have: 

\begin{prop}\label{pr:prodct} If $R$ is an affine $\C$-algebra and $I,\,J$ are ideals of $R$,
then  $I\ct J\ct \inc (IJ)\ct$ \qed\end{prop} 

The following result is proved in the standard-graded case in \cite{Br-cc}.

\begin{thm}\label{thm:highdegct}
Let $R$ be a finitely generated $\N$-graded $\C$-algebra with
$R_0 = \C$ and let $\vect F h \in R$ be elements of positive degrees $\vect d h$.  Suppose
$F$ is homgeneous of degree $d$,  where $d > d_i$,  $1 \leq i \leq h$.  Let
$I = (\vect F d)R$.  Suppose that every element of positive degree has a power in $I$.
Then $F \in I\ct$. 
\end{thm}

\begin{proof}
We may assume that $R$ is reduced. We map a graded polynomial ring 
\[
\C[\vect X n] \surj R,
\] so that \[
R =K[\vect X n]/ \fA,
\]  where $\fA$ is the kernel, and the
map preserves degree.  Let $X_j$ have degree $e_j$.  Define an
action of $\C$ on $\C^n$  by this rule:  if $z = (\vect z n)$,  then let
\[
tz := (t^{e_1}z_1, \, \ldots, t^{e_n} z_n),
\] and let \[
||z|| := \sqrt{\sum_{j=1}^n |z_j|^{2/e_j}}.
\]
Then if $H$ is homogeneous of degree
$\delta$ in the polynomial ring,  $H(tz) = t^\delta H(z)$.  
Moreover,  $||tz|| = |t|\,||z||$.  The action then stabilizes
$X = V(\fA)$.  Let $\z$ be the origin in $\C^n$.  Then $x \in X$, and the
$F_j$ vanish simultaneously only at $x$.  Hence,  $\sum_{i} |F_i|^2$ vanishes
only at $\z$, and we have  \[
1 = \sum_j \frac{\ov{F}_j}{\sum_i |F_i|^2}F_j
\] on $X-\{\z\}$.  
Multiplying by $F$ yields $F = \sum_j g_jF_j$  where the $g_j$ are continuous
on $X - \{\z\}$.  Let $t = ||z||$.  Let $y = t^{-1}z$.  Then  \[
F(z) = F(ty) = t^d F(y)
= t^d \sum_j g_j(y) F_j(y) = \sum_j t^{d-d_j} g_j(y) F_j(ty).
\]      For
$z \neq \z$ in $X$,  define  \[
h_j(z) = ||z||^{d-d_j}g_j(\frac{z}{||z||}).
\]
Then $h_j$ is continuous on $X-\{\z\}$,  and its limit as $z \to \z$  is 0 because
$||z||^{d-d_j} \to 0$, while $g_j$ is bounded on the set $\{y \in X : ||y|| = 1\}$, which is where
$z/||z||$ varies,  since this set is closed and bounded and so compact in the Euclidean topology.  
Since $F$ vanishes at the origin, we  are done.
\end{proof}

\begin{disc}\label{disc:ctx}  We can extend the notion of continuous closure to the local ring $R_\m$ of an affine
$\C$-algebra $R$ at a maximal ideal $\m$ as follows.  Let $I$  be an ideal of $R_\m$.
Let  $X$ be the affine algebraic set $\MaxSpec(R)$ (in the Euclidean topology), let $x \in X$ correspond to $\m$,  
and let $S$ denote the ring of germs of continuous $\C$-valued functions on $X$
at $x$.  Then define $I\ct$ as the contraction of $IS$ to $R_\m$. 

When $I$ is an ideal of $R$,  
we write $I\ctx$ for $(IR_\m)\ct$. \end{disc}

\begin{prop}\label{pr:ctlocal}
Let $R$ be an affine $\C$-algebra and let $X$ be the
corresponding algebraic set.  Let $I$ be an ideal of $R$, and $f \in R$.  Then
$f \in I\ct$ if and only if for all $x \in X$,   $f/1 \in (IR_\m)\ct$,  where $\m$ is the
maximal ideal of $R$ corresponding to $x$. 
\end{prop}
\begin{proof}
Let $\vect fh$ generate $I$.  It is clear that if $f= \sum_{i=1}^h g_i f_i$ with
the $g_i$ continuous on $X$, the equation persists when we take germs at $x \in X$.
For the converse, suppose that $f \in R$ has image in $(IR_\m)\ct$ for all $\m$.
Then for every $x \in X$,  $x$ has a neighborhood $U_x$ in the Euclidean topology on $X$
such that \[
f|_{U_x} = \sum_{i=1}^h g_i^x f_i|_{U_x}
\] on $U_x$, 
where the $g_i^x$ are continuous 
functions on $U_x$.  By making the neighborhoods $U_x$ smaller we may also assume
that the $g_i^x$ are bounded on $U_x$.   The open cover $\{U_x\mid x\in X\}$ has a locally finite refinement
by open sets $V_\lambda$ such that there are continuous $[0,\,1]$-valued functions
$b_\lambda$ on $X$ with the property that $b_\lambda$ vanishes off $V_\lambda$ and 
such that \[
1 = \sum_\lambda b_\lambda.
\] \emph{i.e.}, the $b_\lambda$
give a partition of unity.  Each $V_\lambda$ is contained in some $U_x$ and
so there are continuous  $\C$-valued functions $g^{\lambda}_i$ on each $V_\lambda$,
bounded on $V_{\lambda}$ (obtained by restricting
suitable $g^x_i$), such that 
\[
f|_{V_\lambda} = \sum_{i=1}^h g^{\lambda}_i f_i|_{V_\lambda}.
\]
Then \[
f = \sum_{i=1}^h (\sum_\lambda g^{\lambda}_ib_{\lambda})f_i
\]
and every $\sum_\lambda g^{\lambda}_i b_\lambda$ is a continuous function
on $X$ when defined to be 0 off $V_\lambda$.
\end{proof}

\begin{cor}\label{cor:ctcover}
Let $R$ be a finitely generated $\C$-algebra, with $X$ the associated affine variety.  Let $I \subseteq R$ be an ideal, and $f\in R$.  Let $\{X_j\}_{j \in \Lambda}$ be
an affine open cover of $X$, with $R_j = R[X_j]$.  Then $f\in I\ct$ if and only if $f\in (I R_j)\ct$ for all $j \in \Lambda$.
\end{cor}  

\begin{disc}[ideal closures and gradings]\label{disc:grad}  At this point, we are only aiming to prove Proposition~\ref{pr:cthomog} below, but we eventually
will want to prove similar results for other closure operations where the issue is more difficult.  Let $R$ be a $\Z^h$-graded ring, where $h >0$ is an
integer. Note that this case includes $\N^h$-gradings and, of course, $\N$-gradings.   
If $\alpha = (\vect \alpha h)$ is a $k$-tuple of units of $R_0$,  where the subscript is the zero element in $\Z^h$, there is a degree-preserving
automorphism $\theta_\alpha$ of  $R$ that multiplies forms of degree $(\vect k h)$ by $\alpha_1^{k_1}\,\cdots\, \alpha_h^{k_h}$.  Suppose that
$^\#$ is a closure operation on ideals of $R$ such that the closure of an ideal that is stable under these automorphisms is again stable under these 
automorphisms.  Suppose that $R_0$ contains an infinite field, or, more
generally, that for every integer $N > 0$  that  $R_0$ contains  $N$ units $\vect \alpha N$ such that the elements $\alpha_i - \alpha_j$ for
$i \not= j$  are also invertible.  Then whenever $I$ is homogeneous, its closure $I^\#$ is also homogeneous. By induction on $h$ one can
reduce to the case where $h = 1$.  The result the follows from the invertibility of Vandermonde matrices $\bigl(\alpha_i^{j-1}\bigr)$, where
the $\alpha_i$ are distinct units whose nonzero differences are units. If $f \in I^\#$ is an element whose nonzero homogeneous components
occur in degrees   $d,\, \ldots, \, d+N-1$, it suffices to have $N$ units whose distinct differences are also units in $R_0$ to conclude
that the homoogeneous components of $f$ are in $I^\#$.  
 See Discussion (4.1) in \cite{HHsplit}.  If $R_0$ does not 
have sufficiently many units to carry through the argument, one can seek a family of $\Z^h$-graded $R$-algebras $S_N$  containing $R$ such that 
$R \inc S_N$ preserves degrees and  such that for all homogeneous ideals $I \inc R$ and all $N >0$,  $I^\# \inc (IS_N)^\#$ while $(IS_N)^\# \cap R = I$.  Then, if
$f \in I^\#$,  we have that  $f \in (IS_N)^\#$ for all $N$,  and for $N$ sufficiently large this will imply that all homogeneous components of $f$ are
in $(IS_N)^\#$ and, hence, in $(IS_N)^\# \cap R = I^\#$.  In particular, this method will succeed if one can choose
$S_N$ to be the localization of $R[\vect t N]$ at the element $g_N$ that is the product of all the $t_i$ and all the $t_i-t_j$ for $i \not= j$,  which
may be thought of as $R_0[\vect t N]_{g_N} \otimes_{R_0}R$, and the grading is taken so that the degree zero part is  $R_0[\vect t N]_{g_N}$.
Note that these $S_N$ are smooth and faithfully flat over $R$.  This discussion applies also to the case when $I^\#$ is an ideal
defined ring-theoretically in terms of $I$, even when $^\#$ is not a closure operation, such as inner integral closure, which is treated
in \S\ref{sec:nat}. \end{disc}

If $R$ is a finitely generated $\Z^h$-graded $\C$-algebra, and $\vect \alpha h$ are nonzero elements of $\C$, the automorphisms $\theta_\alpha$ are
$\C$-automorphisms. Hence, if  $I$ is homogeneous,  $I\ct$ is stable under these automorphisms by the persistence of continuous closure 
(Proposition~\ref{pr:pers}).   Therefore, simply because $\C$ is an infinite field in $R_0$, we have:

\begin{prop}\label{pr:cthomog}
Let $R$ be a finitely generated $\Z^h$-graded $\C$-algebra,
and suppose that $\C$ is contained in $R_{0}$,  where the subscript indicates the zero element in $\Z^h$.
Let $I$ be a homogeneous ideal of $R$ with respect to this grading. Then  $I\ct$ is also a homogeneous ideal of $R$ with respect to this grading. \hfill $\square$
  \end{prop}

\section{Seminormal rings}\label{sec:seminor}
In this section we review certain facts about seminormal rings
 and prove that a reduced affine $\C$-algebra 
 is seminormal if and only if every ideal generated by a non-zerodivisor
 is axes closed (equivalently, continuously closed).
 
 Recall \cite{Sw-semi} that a ring $R$ is {\it seminormal} if it is reduced and whenever $f$
is an element of the total quotient ring of $R$ such that $f^2, \, f^3 \in R$,
we have that $f \in R$.\footnote{The concept was introduced by Traverso \cite{Tr-semipic} for a more restricted class of rings.  Swan showed that under Traverso's assumptions, Swan's definition was equivalent to Traverso's.}   

Given a reduced Noetherian ring $R$ with total ring of fractions $\cT$, there is a unique 
smallest seminormal extension $R^{\mathrm{sn}}$ of $R$ within $\cT$,
called the \emph{seminormalization} of $R$. 

For rings containing a field of characteristic 0 the property of being seminormal is equivalent to the property of being weakly normal. 
 See Vitulli's recent survey article \cite{Vit-survey} for a treatment of both notions. (We will revisit the concept of weak normality in \S\ref{sec:bigax}.) 
We collect several facts about seminormality that we will need in the proposition below.  A reference for each part is given with the
statement, except for  (6),  which is immediate from the definition of seminormal, and (7),  which follows at once from  (5) and (6) because
the (strict) Henselization is a directed union of localized \'etale extensions.

\begin{prop}\label{pr:semibasic}
Suppose $R$, $S$ are reduced Noetherian rings.  Let $\icr{R}$ be the integral closure of $R$ in its total ring of fractions. 
\begin{enumerate}
\item $R^{\mathrm{sn}}$ is the set of all $b\in \icr{R}$ such that for any $\p \in \Spec R$, $b/1 \in R_\p + \Jac(\icr{R}_\p)$, where $\Jac$ denotes the Jacobson radical, and $\icr{R}_\p$ is the localization of $\icr{R}$ at the multiplicative set $R\setminus \p$. \cite{Tr-semipic}
\item If $g: R \to S$ is faithfully flat and $S$ is seminormal, then $R$ is seminormal. \cite[Corollary 1.7]{GrTr-semi}
\item If $R$ is seminormal and $W$ is a multiplicative set, then $W^{-1}R$ is seminormal. \cite[Corollary 2.2]{GrTr-semi}
\item Suppose the integral closure of $R$ in its total quotient ring is module-finite over $R$.  The following are equivalent: \cite[Corollary 2.7]{GrTr-semi} \begin{enumerate}
 \item $R$ is seminormal.
 \item $R_\m$ is seminormal for all $\m \in \MaxSpec R$.
 \item $R_\p$ is seminormal for all $\p \in \Spec R$.
 \item $R_\p$ is seminormal for all $\p\in \Spec R$ such that $\depth R_\p=1$.
 \end{enumerate}
\item Suppose $g: R \to S$ is flat with geometrically reduced (e.g. normal) fibers.  If $R$ is seminormal, then so is $S$. \cite[Proposition 5.1]{GrTr-semi}
In particular, if $S$ is smooth over $R$, which includes the case where $S$ is \'etale over $R$, and $R$ is seminormal, then $S$
is seminormal.
\item A directed union of seminormal rings is seminormal.  
\item If $R$ is local and seminormal, then the Henselization of $R$ and the strict Henselization of $R$ are seminormal.
\item Suppose $R$ is excellent and local.  $R$ is seminormal $\iff \hat{R}$ is seminormal. \cite[Corollary 5.3]{GrTr-semi}
\item Let $X$ be an indeterminate over $R$.  $R$ is seminormal $\iff R[\![X]\!]$ is seminormal. \cite[Proposition 5.5]{GrTr-semi}
\item\label{it:multiplier} Let $R$ be a reduced affine $\C$-algebra.  Let $\icr{R}$ be the integral closure of $R$ in its total ring of 
quotients. Let $X$ and $Y$ be the varieties associated to $R$, $\icr{R}$ respectively.  If $\pi: Y \to X$ is the map induced from 
the inclusion $R \inj \icr{R}$, then the seminormalization of $R$ consists of all regular functions $f$ on $Y$ such that $f(y)=f(z)$ 
whenever $y,z\in Y$ are such that $\pi(y)=\pi(z)$. \cite[special case of Theorem 2.2]{LeVit-snwn}
\item Let $R$ be a reduced affine $\C$-algebra,  and let $S$ be the seminormalization of $R$.  Then the map 
of affine algebraic sets corresponding to the inclusion $R \inc S$ is a homeomorphism in both the Zariski and Euclidean 
topologies. \cite[Theorem 1]{AnBom-wn}
\end{enumerate}
\end{prop}

We note that passing to the seminormalization of a ring does not affect continuous closure in the following sense:

\begin{prop}\label{pr:ctcontract}
Let $R$ be a reduced affine $\C$-algebra, and
let $S$ be the seminormalization of $R$.  Let $I$ be an ideal of $R$.
Then $(IS)\ct \cap R = I\ct$.  
\end{prop}

\begin{proof}
By Proposition~\ref{pr:semibasic} (11),
the affine algebraic sets associated with $S$ and $R$ are homeomorphic.  Call
both of them $X$.  The ideals $I$ and $IS$ have the same generators
$\vect f n \in R$.  The condition that $f$ be a continuous linear combination
of these elements is independent of whether we think of the problem over
$R$ or over $S$.
\end{proof}

The following is a characterization of complete local 1-dimensional seminormal rings.  It is based on Traverso's ``glueing'' construction.

\begin{thm}\label{thm:glue}
Let $k$ be a field, let $L_1, \dotsc, L_n$ be finite algebraic extension fields of $k$, and let $(V_i, \m_i)$ be discrete valuation rings such that $V_i / \m_i = L_i$.  Let $S$ be the subring of $\prod_{i=1}^n V_i$ consisting of all $n$-tuples $(v_1, \dotsc, v_n)$ such that there exists $\alpha \in k$ 
such that $v_i \equiv \alpha\  (\md \m_i)$ for all $i$.  Then $S$ is  one-dimensional, local, and seminormal.

Conversely, let $(R,\m, k)$ be a \emph{complete} one-dimensional seminormal Noetherian local ring.  Then there exist such extension fields $L_i$ and such DVRs $V_i$ (which, moreover, are complete) such that $R$ is isomorphic to the ring $S$ described above.
\end{thm}

\begin{proof}
Consider any so-described $S$, and let $W := \prod_{i=1}^n V_i$. Note that $\m = \prod_{i=1}^n \m_i \inc S$, consists of all non-units of $S$, and so
is the unique maximal ideal of $S$.   Let $u \in \m$ be an element of this product
that is nonzero in every coordinate.   Then $u$ is a nonzerodivisor in $S$,  and $uW \inc \m \inc  S$.  It follows that  $W$ is the normalization of  $S$. It is then
clear that $S$ is one-dimensional.  Since $W$ is spanned over $S$ by elements that map to a basis for  $\prod_{i=1}^n L_i$ over $k$,  $W$ is module-finite
over $S$,  and $S$ is Noetherian by  Eakin's theorem: see \cite{Eak-Noeth} or \cite{Nag-Noeth}. Finally, we check seminormality.
Let $0\neq v=(v_1, \dotsc, v_n) \in W$ be
such that $v^2, v^3 \in S$.  Then there exist $\alpha, \beta \in k- \{0\}$ such that 
$v_i^2 \equiv \alpha\ (\md \m_i)$ and $v_i^3 \equiv \beta\ (\md \m_i)$ for all $i$.  Consider the element $\gamma := \beta/\alpha \in k$.   
It follows easily that $v_i \equiv \gamma\ (\md \m_i)$ for all $i$, whence $v\in S$.  Thus, $S$ is seminormal.

Now let $(R,\m,k)$ be a complete one-dimensional seminormal Noetherian local ring.  Let $R'$ be the normalization of $R$.   Note that $R' = \prod_{i=1}^n V_i$, where $(V_i, \m_i, L_i)$ are discrete valuation rings, complete since $R$ is complete.  In particular, if $\p_1, \dotsc, \p_n$ are the minimal primes of $R$, then $V_i = (R/\p_i)'$.  Moreover, since $R'$ is module-finite over $R$, it follows that each $L_i$ is module-finite (i.e., finite algebraic) over $k$.  Let $S$ be as described in the statement of the theorem for these particular $k$, $L_i$, and $V_i$.  Clearly $R$ embeds as a subring of $S$, since for any $r\in R$, the map $R \rightarrow R'$ sends $r \mapsto (\ov{r}_1,\, \dotsc,\, \ov{r}_n)$ (where $\ov{r}_i$ is the residue class of $r$ mod $\p_i$), and the residue class of 
each $\ov{r}_i$ modulo $\m_i$ is clearly the same as the residue class of $r$ mod $\m$, which is, of course, in $k$.  So all we need to show is that the induced injective map from $R$ to $S$ is surjective.

Let $v = (v_1, \dotsc, v_n) \in S$.  Then there is some $\alpha \in k$ such that $v_i \equiv \alpha\ (\md \m_i)$ for all $i$.  Take any $r\in R$ such that $r \equiv \alpha (\md \m)$.  Let $w := v-r = (v_1 - \ov{r}_1, \dotsc, v_n - \ov{r}_n)$.  By construction, $w \in \prod_{i=1}^n \m_i = \Jac(R')$.  
But by part (1) of Proposition~\ref{pr:semibasic}, we have $\m = \Jac(R')$, since $R$ is seminormal. Hence, $w\in \m$, whence $v = r+w \in R+ \m = R$, as was to be shown.
\end{proof}

\noindent \emph{Note:} In Traverso's terminology, $S$ is the \emph{glueing of $R'$ over $\m$}.  That is, $S$ is the pullback of the following diagram of ring homomorphisms: \[
\begin{CD}
  @.        \displaystyle \prod_{i=1}^n V_i\\
  @.        @VVV\\
k @>>> \displaystyle \prod_{i=1}^n L_i.
\end{CD}
\]

\noindent \emph{Note:} The above Theorem may also be deduced from the machinery developed in \cite{Yo-biratint}, although it does not appear explicitly.
\medskip

Let $L$ be an algebraically closed field.  The  notion of a ring of axes over $L$ is defined in the
introduction.  We shall also use the term {\it affine axes ring} over $L$ to
emphasize the distinction from other notions described below.  
By a {\it complete axes ring} over $L$,
we mean a ring of the form \[
L[\![\vect x n]\!]/(x_ix_j: 1\leq i < j \leq n),
\]
where the $x_i$ are formal power series indeterminates.  Such rings are
known to be seminormal.\footnote{Indeed, one can see this as a special case of Theorem~\ref{thm:glue}.}   When the
$x_i$ are indeterminates over an algebraically closed field $L$, we shall refer
to \[
L[\vect x n]/(x_ix_j: 1\leq i < j \leq n)
\]
as a {\it polynomial axes ring} (also called an \emph{affine ring of axes}).  Both complete axes rings and polynomial axes rings are seminormal.  
In fact, we have, we have parts (a) and (b)  of the following proposition
from \cite{Bom-semi},  while parts (c) and (d) follow easily from parts (a) and (b).  
 (See also \cite{Gib-Fp} and \cite{GoWa-Fp} for connections with the notion of $F$-purity.). 

\begin{prop}\label{pr:onedimsemi} Let  $L$ be an algebraically closed field. 
\begin{enumerate}
\item[(a)]  A complete axes ring over $L$ is seminormal.
\item[(b)] Every complete local one-dimensional
seminormal ring of equal characteristic with algebraically closed residue
class field $L$ is isomorphic with a complete ring of axes  over $L$.  
\item[(c)] Every affine ring of axes over $L$ is seminormal.
\item[(d)] A one-dimensional affine $L$-algebra $R$ is seminormal if and only if there
are finitely many \' etale $L$-algebra maps $\theta_i : R \to A_i$,  where the $A_i$ are
affine rings of axes over $L$, and every maximal ideal of $R$ lies under a maximal ideal
of some $A_i$. 
\end{enumerate}
\end{prop}   
\begin{proof}  As already mentioned, parts (a) and (b) are proved in \cite{Bom-semi}.  Part (c)
is then immediate from parts (4)(b) and (8) of Proposition \ref{pr:semibasic}, the definition of affine ring
of axes over $L$,  and the complete case of part (a) above.  For part (d), we first prove ``if." Note that
if  $\widetilde{\m}$ in $A  = A_i$ lies over $\m$ in $R$,  then $R_{\m} \to A_{\widetilde{\m}}$ is
faithfully flat.  Since $A$ is seminormal by part (c),  so is $A_{\widetilde{\m}}$,  and the result
follows from Proposition \ref{pr:semibasic} part (2). To prove ``only if" it suffices to construct a
cover by sets $\Spec(A_i)$ that may be infinite, since we may use the quasicompactness of
$\Spec(R)$ to pass to a finite subcover.  Therefore, it suffices to construct an \'etale extension $A$
that is a ring of axes for each maximal ideal $\m$ of $R$ so that $\m A \not= A$.  If $R_\m$ is regular
we may simply take $A$ to be $R_f$ for a suitable element $f$.  For the finitely many choices of $\m$
such that $R_\m$ is singular, we know the completion of $R_\m$ is a formal ring of axes.  This implies
that there is an \'etale extension $A$  of  $R$ that is an affine ring of axes with $mA \not= A$.  To see why this is true,  let  $S = R_\m$. Let $B$  denote
the normalization of  $S$,  which is semilocal and regular.  The normalization $C'$ of the completion $C$ of
$R_\m$  is the product of the rings $L[\![x_i]\!]$.  Because $S$ is an excellent domain,  $C' \cong  C \otimes_S B$.
Let  $T$ denote the Henselization of  $S$.   Then  $B \otimes_S T$  is module-finite and regular over the Henselian
local ring  $T$, and so is a finite product of discrete valuation domains.  The completions of these give the various
$L[\![x_i]\!]$.  $T$ is a direct limit of \'etale extensions of  $R$ in which there is a unique maximal ideal $\widetilde{\m}$ lying
over $\m$.  For a sufficiently large such extension $A$,  
$B \otimes_R A$  will contain all of the idempotents of $B \otimes_S T$,  and will be regular.    One may localize $A$ at one
element not in $\widetilde{\m}$ so that it is a ring of axes with a unique singularity at $\widetilde{\m}$.  
\end{proof}

\begin{disc}\label{disc:analcomp} Let  $R$ be a finitely generated $\C$-algebra and $\m$ a maximal ideal of $R$.  Map a polynomial
ring  $T = \C[\vect Xn]$ onto $R$ so that the $X_i$ map to generators of $\m$.  Then $\C[\vect Xn] \inc \C\{\!\{\vect Xn\}\!\}$,
the ring of convergent power series in $\vect xn$, and $S  = \C\{\!\{\vect Xn\}\!\} \otimes_T R_\m$ is the {\it analytic
completion} of  $R_\m$.  This ring is a local, excellent, Henselian, faithfully flat extension of $R_\m$, and we have 
$R_\m \inc S \inc \widehat{R_\m}$ 
with the second inclusion faithfully flat as well.  We shall refer to  $\C\{\!\{\vect Xn\}\!\}/(X_iX_j: i \not=j)$ as an {\it analytic axes ring}. \end{disc}

\begin{prop}\label{pr:analax} Let $R$ be a finitely generated $\C$-algebra of dimension one.  Then $R$ is seminormal if and only if for every 
maximal ideal $\m$ of $R$ such that $R_\m$ is not regular,  the analytic completion of $R_\m$ is an analytic axes ring.  \end{prop}
\begin{proof} We know that $R$ is seminormal if and only if each $\widehat{R_\m}$ is, and this is automatic if $R_\m$ is regular
(which includes any isolated points).  It is therefore sufficient to show that the analytic completion $(A, \m_A)$ of $R_\m$  is an analytic
ring of axes if $\widehat{R_\m} =  \widehat{A}$ is a formal ring of axes.  Since $A$ is one-dimensional excellent and
Henselian,  its minimal primes $\vect P n$ correspond bijectively to those of $\widehat{A}$ via expansion and contraction.  
Let $Q_i = \bigcap_{j \not=i} P_j$.  Then $Q_i\widehat{A}$ is a principal ideal generated by an element not in $(\m_A\widehat{A})^2$,
 $Q_iQ_j = 0$ for $i\not=j$, and $\sum_i Q_i\widehat{A} = \m_A\widehat{A}$,  from which it follows that $Q_i$ is a principal ideal
generated by an element not in $\m_A^2$,  that $Q_iQ_j = 0$ for $i \not= j$ and that $\sum_i Q_i = \m_A$.  Let $x_i$ generate $Q_i$.
Then the $x_i$ are a minimal set of generators of $\m_A$,  $x_ix_j = 0$ for $i\not=j$, and since $P_i \widehat{A}= \sum_{j \not=i} Q_j\widehat{A}$,
we have that $P_i$ is generated by $\{x_j: j\neq i\}$ for each $i$.  The map $\C[\vect X n] \to A$ (sending $X_i \mapsto x_i$ for all $i$) 
induces a $\C$-homomorphism $\theta$ of $B =  \C\{\!\{\vect Xn\}\!\}/(X_iX_j: i \not=j)$ to $A$ such that the map of completions is an isomorphism.
Thus, this map is injective. If  $\cP_i$ is generated by the $X_j$ for $j \not= i$,  
then $B/\cP_i \to A/P_i$ is a map from $\C\{\!\{X_i\}\!\} \to A/P_i \cong \C\{\!\{x_i\}\!\} \inc A$ which induces an isomorphism of completions, and
so must be an isomorphism.  Since $A$ is the sum of the subrings $\C\{\!\{x_i\}\!\}$,  $\theta$ is surjective.
\end{proof}

Both in this section and the next we shall need to use Artin approximation to
descend a map from  an affine K-algebra to a complete ring of axes over an extension
field L  of  K   to a map to an \'etale extension of a polynomial ring of axes, which
will be seminormal.  Specifically:

\begin{thm}[descent via Artin approximation]\label{thm:descent}
Let $K$ be an algebraically
closed field, let $L$ be an extension field, let $R$ be an affine $K$-algebra,
let $I$ be an ideal of $R$,  and let $f$, $g$ be elements of $R$.  Let
$R \to S$ be a $K$-algebra homomorphism to a complete ring of axes
$S$ over $L$ such that the image of $g$ is not in $IS$.  Then there is
a $K$-algebra homomorphism $R \to S_0$, where $S_0$ is an \'etale
extension of a polynomial ring of axes over $K$, such that the image of
$f$ is not in $IS_0$.  Moreover, if the image of $f$ is not a zerodivisor
in $S$,  the map $R \to S_0$ may be chosen to satisfy the additional
condition that the image of $f$ is not a zerodivisor in $S_0$.
\end{thm}

\begin{proof}
$S$ is a complete
ring of axes with algebraically closed residue class field $\cL$, and such
a ring is the completion $T$ of a polynomial axes ring
$\cL[\vect xn]/(x_ix_j: i < j)$ localized at the maximal ideal $\m$ generated by
the $x_j$.   Call the localized ring $T_0$. 
Moreover, we can choose $N$ so large that the image of $f$
is not in $IT + \m^NT$.  Think of $R$ as $K[\vect y s]/(\vect g h)$.  Then
the images $z_j$ of the $\vect y s$ in $T$ give solutions of the equations $g_j = 0$
in $T$,  and we may use Artin approximation, \emph{i.e.}, the main result of  \cite{Ar-approx}, 
to find a solution $\vect {z'} s$ of these equations 
in the Henselization $T_0\h$ of $T_0$ congruent to the $z_j$ modulo $\m^NT$. 
We can map $R \to T_0\h$ as a $K$-algebra so that the images of the $y_j$ map to
the $z'_j$ and we still have that $f \notin (I +\m^N)T_0\h$.  In particular,
$f \notin IT_0\h$.  Since $T_0\h$ is a direct limit of finitely generated \'etale
extensions of $B = \cL[\vect xn]/(x_ix_j)$,  we have a finitely generated
\'etale extension $C$ of $B$ and a $K$-algebra map $R \to C$ such
that $f \notin IC$.  

Hence,  $C = B[\vect W m]/(\vect G m)$ is such that the image of the Jacobian
determinant $\det \bigl(\partial G_j/\partial W_j\bigr)$ is a unit of $C$.  

Now let $A$ denote a varying but finitely generated $K$-subalgebra over
$L$ sufficiently large to contain all the coefficients of the $G_j$.
Let $$B_A = A[\vect x n]/(x_ix_j:i < j)$$ and let $$C_A = B_A[\vect W m]/(\vect G m).$$
Then $C$ is the direct limit of the rings $C_A$, and so the Jacobian determinant
is invertible in $C_A$ for all sufficiently large $A$.  We may therefore chose
$A$ so large that $C_A$ is \'etale over $B_A$ and the map $R \to C$ factors
$R \to C_A \to C$.  If $D$ is any $A$-algebra we write $B_D$ and $C_D$
for $D \otimes_A B_A$ and $D \otimes_A C_A$,  respectively.  Let $\cK$ be
the fraction field of $A$.  Then $C = C_{L} \cong L\otimes_{\cK} C_{\cK}$
is faithfully flat over $C_{\cK}$.   Hence,  $f \notin IC_{\cK}$.  

We have the exact sequences
 $$0 \to I_A \to C_A \to C_A/I_A \to 0,$$ 
$$0 \to fC_A + IC_A \to C_A \to C_A/(fC_A + IC_A) \to 0, \hbox{\ \ and}$$ 
$$0 \to IC_A \to (fC_A + IC_A) \to W_A \to 0,$$
where $W_A$ is a cyclic $C_A$ module spanned by the image of $f$.
By the lemma of generic freeness (cf.\ \cite[Theorem 24.1]{Mats}, \cite[Lemma 8.1]{HoRo-invar}), we can localize
at one nonzero element $a \in A$ so that all of the modules in these sequences become
$A_a$-free.  We change notation and continue to write objects with the subscript
$A$:   $A$ has been replaced by $A_a$.  The $A$-free module $W_A$ is not
0, since this is true even after we apply $\cK \otimes_A \blank$.  Let $\mu$ be
any maximal ideal of $A$.  Then $A/\mu = K$, and we use the subscript $K$
to indicate these various algebras and modules after tensoring with $K = A/\mu$
over $A$.  We have a map $R \to C_K$. 
Because $W_K \not=0$,  we have that the image of $f$ is not
in $IC_K \inc C_K$.  But $C_K$ is a finitely generated \'etale extension of
$B_K$ which is a polynomial ring of axes over $K$.

We now consider the modifications needed to preserve the condition that the
image of $f$ be a non-zerodivisor.  Hence,  the image of $f$ in
$L[\![\vect x n]\!]/(x_ix_j)$ is not in any of the ideals $P_i$ generated by all
of the $x_j$ for $j \not=i$.   Then $f \notin P_i + \m^N$ for sufficiently large $N$.
Thus, we may choose $N$ so large that when we apply Artin approximation,
the image of $f$ in the Henselization  $T_0\h$ is not in any of the ideals
$P_i \cap T_0\h$.  When we replace the Henselization by a finitely generated
\'etale extension $C$ of $B$, multiplication by the image of $f$ may have a kernel, but
the kernel will be killed by localization at one element of $C$ not in the contraction
of the maximal ideal of $T_0\h$.  Thus, we may assume that the image of $f$ is not a 
zerodivisor on $C$.  For the last step in the descent, when we pass from $C_A$
to $C_{A/\mu}$,  we need to preserve the exactness of the sequence
$0 \to C_A \arrow{\phi} C_A \to C_A/\phi C_A \to 0$,  where $\phi$ is the image
of $f$,  when we apply $(A/\mu)\otimes_A \blank$.  We can do this by localizing
at one element of $A-\{0\}$ so that all of the terms of the sequence become $A$-free.
\end{proof}

\section{Axes closure and one-dimensional seminormal rings}\label{sec:axes}

We want to extend the notion of axes closure to a larger class of rings.  We first note:

\begin{thm}\label{thm:axsemi}  Let $R$ be a Noetherian ring and let $I$ be an ideal of $R$.
Then statements (1), (2), and (3) below are equivalent. 

Moreover, if $R$ is a finitely generated algebra over a field $K$ of characteristic 0, 
and $L$ is the algebraic closure of $K$, then the first six of the statements below are equivalent.
 If we also assume that $K = L = \C$,  then all seven of the statements below are equivalent.

 \begin{enumerate}
\item For every map from $R$ to an excellent one-dimensional seminormal
ring $S$,  the image of $f$ is in $IS$.
\item For every map from $R$ to an excellent local one-dimensional seminormal
ring $S$,  the image of $f$ is in $IS$.
\item For every map from $R$ to a complete local one-dimensional seminormal
ring $S$,  the image of $f$ is in $IS$.
\item For every map from $R$ to a complete local one-dimensional seminormal
ring $S$ with algebraically closed residue field, $f \in IS$.  
\item For every $K$-algebra  map from $R$ to a complete  axes ring $S$ with
residue class field $L$,  $f \in IS$.
\item For every $K$-algebra map from $R$ to a finitely generated \'etale extension
$S$ of a polynomial axes ring over $L$,  $f \in IS$.
\item  For every $\C$-algebra map $\theta$ from $R$ to an analytic ring of axes $(S,\n)$ over $\C$ such
that $\theta^{-1}(\n)$ is maximal in $R$, $f \in IS$.
\end{enumerate}
\end{thm}

\begin{proof}
 $(1) \imp (2) \imp (3) \imp (4) \imp (5)$ is obvious.  If $R$ is only assumed Noetherian
 we can prove $(3) \imp (1)$ as follows.  Suppose that there is a map $R \to S$, where
 $S$ is one-dimensional and seminormal, such that $f \notin IS$.  Then this can be
 preserved when we localize at some maximal ideal of $S$ and then complete.  The
 completion of $S_\m$ is still seminormal.  In the rest of the proof we assume the
 additional hypothesis on $R$.  
 
 To see that $(5) \imp (6)$:
If (6) fails, we also have $f \notin IS_\m$ for a local ring $S_\m$ of $S$
and also for the completion $T$ of $S_\m$.  $S$ is seminormal, and hence so is
$T$, which will have algebraically closed residue class
field $L$.    $T$ is a complete axes ring;  see Proposition~\ref{pr:onedimsemi}.

We next show that $(6) \imp (1)$.  Assume (6) and suppose that (1) fails
for $S$.  Then it also fails for some localization of $S$ at a maximal ideal, and for the completion
of that ring.  Hence, we may assume that $S$ is complete local.  By  part (7) of 
Proposition~\ref{pr:semibasic}, we may replace $S$ by its strict Henselization, which will 
have algebraically closed residue class field since we are in equal characteristic 0,  
and we may then complete again.  We may therefore assume that $S$ is a complete
ring of axes with algebraically closed residue class field $\cL$.  Extend $K$
to a coefficient field for $S$,  which we also denote $\cL$.  Then we have
$K \inc L \inc \cL \inc S$.  We may now replace $R$ by $L \otimes_K R$: 
we still have a map from this ring to $S$,  using the embedding $L \inj \cL$.
The result now follows from Theorem~\ref{thm:descent}.

Finally, it is clear that $(1) \imp (7)$, and it will suffice to prove that  $(7) \imp (6)$. Suppose that we have 
a map $R \to A$  where $A$ is \'etale over a polynomial ring of axes, such that $f \notin IA$. Then $A$ is a one-dimensional seminormal 
ring, and we can preserve that $f \notin IA$ by localizing at some maximal ideal $\mu$ of $A$.  The inverse image will be a maximal ideal 
of $R$, since $R$ and $A$ are affine $\C$-algebras.  Let $(S,\, \n)$ be the analytic completon of $A_\mu$.   Note that $\widehat{S} \cong
\widehat{A_\mu}$. Then $A$, $A_\mu$, $\widehat{A_\mu}$, and $S$ are seminormal by parts (5), (3), and (8) of Proposition~\ref{pr:semibasic} 
and $S$ is an analytic axes ring by Proposition~\ref{pr:analax}.  Since $S$ is faithfully flat over $A_\mu$ and $f \notin I A_\mu$, 
it follows that $f \notin IS$, which completes the proof.
\end{proof}

\begin{cor}\label{cor:axsemiC}
If $R$ is an affine $\C$-algebra, the equivalent conditions
(1) through (6) of Theorem~\ref{thm:axsemi} hold if and only if $f \in I\ax$. 
\end{cor}
\begin{proof} If (1) through (6) hold, it is clear that $f \in I\ax$,  since  rings of axes are 
one-dimensional excellent seminormal rings. Suppose that $f \in I\ax$.  It suffices
to verify condition (6).  Since $L = K = \C$,  by \ref{pr:semibasic} it suffices to show
that for every map to a one-dimensional affine seminormal ring $S$ over $R$,  the image
of $f$ is in $IS$.  If not, we can choose a maximal ideal $\m$ of $S$ such that 
$f \notin IS_\m$.  By part (d) of Proposition \ref{pr:onedimsemi} 
there is an \'etale map  $S \to A$ such that $A$ is an affine
ring of axes over $\C$  and has a maximal ideal $\widetilde{\m}$ lying over $\m$.  Since
the image of $f$ is not in $IS_\m$ and $S_\m \to A_{\widetilde{\m}}$ is faithfully flat, we have that 
the image of $f$ is not in $IA_{\widetilde{\m}}$ and, hence, not in $IA$. 
 \end{proof}

\begin{defn}\label{def:axsemi}  Theorem~\ref{thm:axsemi} shows that
conditions (1), (2), and (3) are equivalent for any Noetherian ring $R$.  We therefore
define $f \in R$ to be in the {\it axes closure} $I\ax$ of $I$ in the general case if these
three equivalent conditions hold. Once we have made this definition, we have
at once:
\end{defn}

\begin{cor}\label{cor:axsemi}
Let $R$ be a finitely generated $K$-algebra, where $K$
is a field of characteristic $0$, $I \inc R$ an ideal, and $f \in R$. 
Then $f \in I\ax$ if and only if the equivalent
conditions (1) through (6) of Theorem~\ref{thm:axsemi} hold.
\end{cor}

The following remark is obvious from the definition, but is, nonetheless, quite important.

\begin{prop}\label{pr:onedimax} Let $R$ be an excellent one-dimensional seminormal ring.
Then every ideal of $R$ is axes closed. \qed\end{prop}  

\begin{prop}\label{pr:normthenaxthensemi}
 Let $R$  be a Noetherian ring.
 \begin{enumerate}
 \item The axes closure of an ideal is contained in its integral closure.
 \item If $R$ is normal domain, every principal ideal is axes closed.
 \end{enumerate}
\end{prop}

\begin{proof}
One may test integral closure by mapping to Noetherian valuation domains, and
these may be replaced by their completions, which are excellent.  The second statement follows, since
principal ideals are integrally closed in a normal domain. \end{proof}

\begin{lemma}\label{lem:prod}
If $I, \, J \inc R$,  then $I\ax J\ax \inc (IJ)\ax$. In particular,
if  $r \in R$,  $r(I\ax) \inc (rI)\ax$.
\end{lemma}
\begin{proof}
Let $\theta: R \to B$ denote any homomorphism to an excellent one-dimensional seminormal ring.
The first statement follows since $I\ax J\ax B = (I\ax B)(J\ax B) = (IB)(JB)
= (IJ)B$, and the second statement follows from the case where $J = rR$. 
\end{proof}

In \S\ref{sec:natax} we prove that an excellent ring is seminormal if and only if every principal
ideal generated by a non-zerodivisor is axes closed:  see  Theorem~\ref{thm:semiaxes}

\begin{example}   
We give an example of a reduced finitely generated
$\C$-algebra of pure dimension two with precisely two minimal primes
which is seminormal, although one of its quotients
by a minimal prime is not seminormal.  Its normalization is the product of two
polynomial rings in two variables. It has a principal ideal generated by
a zerodivisor that is not axes closed.  This example is similar to \cite[Example 2.11]{GrTr-semi}.

Let $S = \C[u, \, v] \times \C[x, \, y]$. Let $R$ be the subring
of $S$ generated over $\C$ by $q = (u^2, \, x)$,  $r = (u^3,\,y)$, $s = (v,\,0)$,
and  $t = (uv,\,0)$.  Then $w = q^3 - r^2 = (0,\, x^3 - y^2) \in R$, and
$z = v+w = (v,\, x^3 - y^2)$ is a non-zerodivisor in $S$ and, hence, in $R$.
Thus,  $e = (1, \, 0) \in S$ is in the total quotient ring of $R$,  since it is integral
over $R$ and $ze = (v,\,0) \in R$,  $(u,\,0)$ is integral over $R$ since
its square is $qe$,  which is integral over $R$,  and it is in
the total quotient ring of $R$,  since $z(u,\,0) \in R$.  It follows that
the integral closure of $R$ is $S$.

Then $R$ consists of all pairs of the form   
$\bigl(P(u^2,\,u^3) + vH(u,v),\, P(x,\,y)\bigr)$ where
$H(u,\,v)$ is an arbitrary polynomial in $u,\,v$ and $P$ is an
arbitrary polynomial in $x,\,y$.  Alternatively,  $R$ consists of
all pairs $\bigl(Q(u,\,v),\, P(x,\,y)\bigr)$ such that 
$Q(u,v) \equiv P(u^2,\, u^3)$ mod $v\C[u,v]$.  From the latter
description we see that $R$ is seminormal, for if
$Q(u,\,v)^2 \equiv P(u^2,\, u^3)^2$  mod $v\C[u,v]$ and
$Q(u,\, v)^3 \equiv P(u^2, \, u^3)^3$ mod $v\C[u,v]$,
then $Q(u,v) \equiv P(u^2,\, u^3)$ mod $v\C[u,v]$.  

The two minimal primes of $S$ contract to incomparable
primes $P = (\C[u,v] \times 0) \cap R$ and  $Q = (0 \times \C[x,y]) \cap R$.
Clearly,   $P \cap Q = 0$.  Note that $(v,0), \, (uv,\,0) \in P - Q$ and
$(0,\, x^3 - y^2) \in  Q-P$.  Hence,  $P$ and $Q$ constitute all the 
minimal primes of $R$.   We have that $R/P \cong K[u^2, u^3, uv, u]$,
which is not seminormal,  while $R/Q \cong K[x,y]$.  

The maximal spectrum of $R$, in the Euclidean topology,  is the union of two complex planes,
$\MaxSpec(\C[u^2, u^3, uv, v])$,  which may be identified topologically
with $\MaxSpec(\C[u,v]) = \C^2$,  and $\MaxSpec(\C[x,y])$.  These meet along
a topological line which may be identified with $\cV(v)$ in the first
plane and with $\cV(x^3 - y^2)$ in the second plane.   

 Then $u$ represents a continuous function on $\C^2$ that may be restricted
to the intersection and so viewed as a function on the closed set
$\cV(x^3 - y^2)$ in the second plane.  Hence, by the Tietze extension theorem 
there is a continuous function $\theta$ on the second plane that extends the 
restriction of $u$. The pair $(u,\theta)$ represents a continuous $\C$-valued 
function on $\Spec(R)$, and we have that $(uv, 0) = (u,\theta)(v,0)$.  It follows 
that $(uv,0)$ is in the continuous closure of $(v,0)R$  in $R$,  and, hence,
in the axes closure.  But it is not in the ideal.  
\end{example}

In \S\ref{sec:natax} we prove that in a reduced affine $\C$-algebra, axes closure and
continuous closure agree for principal ideals generated by a non-zerodivisor.  
We do not know whether they agree for principal
ideals generated by a zerodivisor.

The following result, together with Proposition~\ref{pr:ctlocal} and Corollary~\ref{cor:ctcover}, 
gives one sense in which the issue of whether axes closure
and continuous closure agree in affine $\C$-algebras is local.

\begin{prop}\label{pr:axlocal}
Let $R$ be a Noetherian ring, $I$ and ideal of $R$,
and $f \in R$.
\begin{enumerate}[label=(\alph*)]
\item  If $\varphi: R \to S$ is any homomorphism and $f \in I\ax$,  then $\varphi(f) \in (IS)\ax$. (In other words, axes closure is \emph{persistent}.)
Hence, the contraction of an axes closed ideal is axes closed.
\item $f\in I\ax$ if and only if for each $P \in \Spec R$, one has $f \in (I R_P)\ax$.
\item $f \in I\ax$ if and only if this holds in an affine open neighborhood of each prime
ideal of $R$.  In particular, $I$ is axes closed if and only if this is true for an affine
open neighborhood of each prime of $R$.
\end{enumerate}
\end{prop}  

\begin{proof}
The first statement in (a) is immediate from the definition, and the second
statement follows at once from the first statement.  

For part (b): Part (a) guarantees that if $f \in I\ax$, this remains true when we localize.  It will suffice to show
that if $f \in (IR_P)\ax$ for all $P$ then $f \in I\ax$.  If not, we can map to a
one-dimensional local seminormal ring $(S,\,Q)$ such that $f \notin IS$.  
But then $f$ is not in $(IR_P)\ax$,  where $P$ is the contraction of $Q$.

Part (c) follows at once.    
\end{proof}

\begin{prop}\label{pr:axcolon} Let $R$ be a Noetherian ring, and $I$ an ideal. Let $f \in R$,  let $J$
be an ideal of $R$,  and let $W$ be a nonempty multiplicative system in  $R$.
If $I$ is axes closed, then so are  $(I:_Rf)$,  $(I:_RJ)$,  $\bigcup_{n=1}^\infty (I:_RJ^N)$,  and the
contraction $\fA = \{r \in R: \hbox{\ for some\ } w \in W,  wr \in I\}$ of $IW^{-1}R$  to  $R$.  \end{prop}  
\begin{proof}  Suppose  $g \notin (I:_R f)$.  Then  $fg \notin I$,  and we can choose a homomorphism
$h:R \to A$,  where $A$ is a one-dimensional excellent seminormal ring, such that  $h(fg) \notin IA$.
But then  $h(g) \notin (IA:_A h(f)) \supseteq (I:_R f)A$,  which shows that $g \notin (I:_R f)\ax$. 
This establishes the first statement.

Since $I:_R J = \bigcap_{j \in J} (I:_R j)$,  the second statement follows.  The ideals  $(I:_R J^N)$  form an ascending
chain, and so the union is equal to one of them.  This proves the third statement.  Finally,  $\fA = \bigcup_{w \in W} (I:_R w)$.
The union is directed, since  $(I:_R v) \cup (I:_Rw) \subseteq (I:_R vw)$.  The family of ideals $\{(I:_R v) \mid v \in W\}$ 
therefore has a maximal element, which must be maximum,  and the union consequently has the 
form  $(I:_R w)$.  \end{proof}

\begin{prop}\label{pr:intprimax}  In any Noetherian ring $R$, every axes closed ideal is 
an intersection of primary axes closed ideals.   If  $R$ is an affine $\C$-algebra, these
may be taken to be primary to maximal ideals.  \end{prop}  
\begin{proof}  Let $f \notin I$,  where $I$ is axes closed.   Then we can choose a map
$h: R \to A$, where $A$ is one-dimensional, excellent, and seminormal, and an  ideal
$I$ of  $A$ such that  $h(f) \notin IA$.  In any Noetherian ring, every ideal is an intersection
of ideals $J$  that are primary to maximal ideals. We can choose such a primary
ideal $J$ of $A$ so that $h(f) \notin J$,  and then $h^{-1}(J)$ will be primary and axes closed with 
$f \notin h^{-1}(J)$.   If  $R$ is an affine $\C$-algebra, we may take $A$ to be an affine
ring of axes over $\C$.  In this case, the inverse image of the radical of $J$,  which is
a  maximal ideal of $A$, is  a maximal ideal $m$ of  $R$,  and the inverse image of $J$ is
primary to $m$.   
\end{proof}

\begin{prop}\label{pr:primloc}  Let $\m$ be a maximal ideal of a Noetherian ring $R$,  and let $I$ be primary
to $\m$ and axes closed.  Then the expansions of $I$ to $S = R_\m$ and to $S = \widehat{R_\m}$  are axes closed.  This
is also true if $S$ is the Henselization of $R_\m$, or, when $R$ is a finitely generated algebra over $\C$,  the
analytic completion of $R_\m$.  
\end{prop}
\begin{proof}  Consider the case where $S =\widehat{R_\m}$.  Let  $f \in S$ be an element of $S - IS$.  Then, since
$R/I \to S/IS$ is an isomorphism,  we may choose  $g \in R$ such that $g \equiv f$ mod $IS$. By Definition~\ref{def:axsemi} 
and Theorem~\ref{thm:axsemi},  we can choose
a map from $R \to A$ where $A$ is a complete local seminormal ring of dimension 1 such that  $g \notin IA$.  Since $I$
maps into the maximal ideal $\m_A$ of $A$ (or else $IA$ would be the unit ideal),  we have that $\m$ maps into $\m_A$.
But then the map extends continuously (in the $\m$-adic and $\m_A$-adic topologies) to a map  $S = \widehat{R_\m} \to A$.  
Since the image of $g$ is not in $IA$ and $f \equiv g$ mod $IS$,  we have that the image of $f \notin IA$. Thus, 
$IS$ is axes closed.  In the other cases, if $f$ were in the axes closure of $IS$, it would be in the axes closure of $I\widehat{R_\m}$
as well,  and we know this is $I\widehat{R_\m}$.  But in every instance $S \to \widehat{R_m}$ is faithfully flat, so that
the contraction of $I\widehat{R_\m}$ to $S$ is $IS$.  
\end{proof}

The following result is a weak result on the compatibility of axes closure with smooth base change. It suffices to prove that axes closures
of homogeneous ideals are homogeneous.  See also Theorem~\ref{thm:axsmooth}, which is a much more difficult
result on compatibility of axes closure with smooth base change.

\begin{prop}\label{pr:axsmooth} Let $R$ be Noetherian and let $S$ be faithfully flat, essentially of finite type, and smooth over $R$.  Then for
every ideal $I \inc R$,  $(IS)\ax$  contracts to $I\ax$ in $R$.  \end{prop} 
\begin{proof}  Suppose that $f$ is in the contraction of $(IS)\ax$ but not in $I\ax$.  Then we can choose a homomorphism $R \to A$,
where $(A, \, \m_A)$ is a complete local one-dimensional seminormal ring, such that  the image of $f$ is not in $IA$.  
Then $S_A = S \otimes_R A$ is faithfully flat, essentially of finite type,
and smooth over $A$.  By part (5) of Proposition~\ref{pr:semibasic},  $S_A$ is seminormal, and it is excellent.  If we localize at a minimal prime
$Q$ of $\m_A S_A$ in $S_A$,  we obtain a one-dimensional, seminormal, local, faithfully flat, excellent extension $B$ of $A$.  Since the image of 
$f$ is not in  $IA$,  we have
that the image of $f$ is not in $IB$.  But since  $R \to B$ factors $R \to S \to S_A \to B$,  this contradicts the assumption that
$f$ is in $(IS)\ax$. \end{proof}

\begin{prop}\label{pr:axgrad} Let $R$ be Noetherian, and suppose that $R$ is $\Z^h$-graded, where $h \geq 1$.  Let $I$ be a homogeneous
ideal of $R$ with respect to this grading.  Then $I\ax$ is also homogeneous with respect to this grading. \end{prop}
\begin{proof}  We want to apply Discussion~\ref{disc:grad}. The only difficulty is that $R$ may not have sufficiently many units.  
But each of the rings $S_{g_N}$ constructed in Discussion~\ref{disc:grad} is finitely presented, faithfully flat, and smooth over $R$,
and so the result follows from that discussion and Proposition~\ref{pr:axsmooth} just above.  \end{proof}

\section{Special and inner integral closure, and natural closure}\label{sec:nat}
Let $R$ be a Noetherian ring, $I \inc R$, and $r \in R$.  
The following conditions are well known to  be equivalent:

\begin{enumerate}
\item There is a monic polynomial  $f(x) \in R[x]$ of some degree
$d$ such that the coefficient of $x^{d-j}$ is in $I^j$,  $1 \leq j \leq d$, and such that $f(r)=0$.
\item For every map $R \to V$,  where $V$ is a DVR,  
$rV \inc IV$.
In other words, $r$ has  order at least as large as $IV$ under the
valuation associated with $V$.
\end{enumerate}

The elements satisfying these equivalent conditions form an ideal
$\oI$ called the {\it integral} closure of $I$.

The {\it special part of the integral closure} of $I$ \cite{nme-sp} is defined when $(R,\,\m)$ is local
and $I \inc \m$.  It consists of all $r$ which satisfy a monic polynomial
as in (1) such that the coefficient of $x^{d-j}$ is in $\m I^j$  for $1 \leq j \leq d$.
The special part of the integral closure is an ideal containing $\m I$ and contained in
$\oI$ but it typically does not contain $I$.  

More generally, for any ideal $J$ of a Noetherian ring $R$ we can define the
$J$-special integral closure $\spi I J$ of $I$ to consist of all elements $r$ in $R$
that satisfy an $J$-special polynomial over $I$:  this means that the polynomial
is monic of degree $d \geq 1$ and the coefficient of $x^{d-j}$ is in $JI^j$
for $1 \leq j \leq d$.  We shall soon see that the condition depends only
on $\Rad(J)$,  and not on $J$ itself.  Our main interest is in the cases
where $J = I$ or $J = \m$.    

Note that while the integral closure of a Noetherian domain need not be
Noetherian, it is still true that it is a Krull ring:  principal ideals have only finitely
many minimal primes, the localization at such a minimal
prime is a DVR, and one has primary decomposition for principal ideals.  
Cf.\ \cite[pp.\ 115-117 and Theorem (33.10) on p.\ 118]{NagLR}. 

If  $I = 0$ or $J = 0$,  then $\spi I J$ is the ideal of nilpotent elements of $R$, and is
$(0)$ if $R$ is a domain.

\begin{thm}\label{thm:testcl}
Let $(R,\,\m)$ be a Noetherian domain, let $I$ and $J$ be nonzero ideals of $R$,
and let $r \in R$.  The following conditions are equivalent.
\begin{enumerate}
\item $r$ is in the $J$-special integral closure of $I$.
\item\label{it:powers} There exists an integer $n \geq 1$ such that $r^n \in \icip{JI^n}$. In this
case,  all multiples $tn$ of $n$ have the same property:  in fact,  $r^{tn} \in \icip{J^tI^{tn}}$. 
\item There exists an integer $n \geq 1$ such that for all maps 
$R \inj V$ where $V$ is a DVR  then \[
\ord_V(r) \geq  \ord_V(J)/n + \ord_V(I).
\]
\item For all maps $R \inj V$ where $V$ is a DVR,  $\ord_V(r) \geq \ord_V(I)$, and the inequality
is strict if $\ord_V J > 0$.  
\item If $R$ is a domain, $I = (\vect f h)R$ and $J = (\vect g k)R$ with the elements
$f_i$, $g_j$ all nonzero, it suffices that the condition in (4) hold when $V$ is the localization
of the normalization of  $R[I/f_i][J/g_j]$ at one of the minimal primes of the $f_i$, $g_j$ for all $i,\,j$ and
choices of the minimal prime.  If  $J = I$,  we may use instead the minimal primes of
the  $f_j$ in the normalizations of the rings $R[I/f_i]$.
\end{enumerate}
\end{thm}

\begin{proof}
The second statement in (2) is clear.
$(2) \imp (1)$ since the equation showing integral dependence for $r^n$ on  $JI^n$
may be viewed as an equation that $r$ satisfies, and this provides the $J$-special polynomial over $I$.  We prove $(1) \imp (3)$.  
Suppose (1) holds with a $J$-special polynomial over $I$ of degree $n$.
Then for some $j$, $1 \leq j \leq n$,  \[
\ord_V(r^n) \geq \ord_V(JI^j) + \ord_V(r^{n-j})
\] and
so \[
j\,\ord_V(r) \geq \ord_V(J) + j\,\ord_V(I)
\] and \[
\ord_V(r) \geq \ord_V(J)/j + \ord_V(I).
\] Hence, we have
\[
\ord_V(r) \geq \ord_V(J)/n + \ord_V(I),
\]  no matter what $j$ is.  

$(3) \imp (4)$ is clear, and $(4) \imp (5)$ is clear.  

Therefore, the proof will be complete if we can show
that $(5) \imp (2)$.  Choose a value of $n$ so large that condition (3) holds with this
value of $n$  for all of the finitely many valuation rings described in (5).  Now consider
any injection  $R \inj V$ where $V$ is a DVR.  Then $IV$ is generated by the image of
some $f_i$, and $JV$ is generated by some $g_j$.  For these two elements,
the map $R \inj V$ factors $R \inj R[I/f_i][J/g_j] \inj V$ and, hence, 
$R \inj S_{ij} \inj V$,  where $S_{ij}$ is the normalization of $R[I/f_i][J/g_j]$.
We claim that in $S_{ij}$,  $r^n \in JI^nS_{ij} = g_jf_i^nS_{ij}$.  Since
$S_{ij}$ is a Krull domain,  it suffices to see this after localizing at
each of the minimal primes of $g_jf_i^nS_{ij}$. Since each of these
is a minimal prime of $g_j$ or $f_i$, this produces
one of the discrete valuation rings
$W$  for which we have assumed that \[
\ord_W(r) \geq \ord_W(J)/n + \ord_W(I).
\]
Multiplying by $n$  gives the result we need.  Since  $r^n \in JI^nS_{ij}$, this
continues to hold in $V$.  But then $r^n \in \icip{JI^n}$, as required.

When $J = I$,  both are generated by the image of some  $f_i$  after expanding to $V$,
and so the map $R \inj V$ factors $R \inj R[I/f_i] \inj V$,  and the rest of the proof is the same.
\end{proof}

\begin{cor}
For any Noetherian ring $R$ with ideals $I$ and $J$, $r \in R$ is in the
$J$-special integral closure of $I$ if and only if for some $n$,  $r^n \in \icip{JI^n}$.  This may
be tested modulo every minimal prime $P$ of $R$.  

The $J$-special integral closure of $I$ is an ideal,  depends only
on $\Rad(J)$, and  lies between $\Rad(J)I$ and $\Rad(J)\cap \oI$.   
\end{cor}
\begin{proof}
Say the minimal primes of $R$ are $P_1, \ldots, P_h$.   If $r^{n_i} \in \icip{JI^{n_i}}$ mod $P_i$
for $1 \leq i \leq h$, and $N$ is the product of the $n_i$, by Theorem~\ref{thm:testcl} (2)
we have that $r^N \in \icip{JI^N}$ mod $P_i$ for all $i$, and this implies
$r^N \in \icip{JI^N}$ in $R$.  If we have $J$-special equations over $I$ satisfied for $r$ mod each $P_i$,  the
value of the product of these on  $r$ is nilpotent,  and so a power of the product
will be 0.  Thus, the equivalence reduces to the domain case, where we already know it. 

In the domain case,  $\spi I J$ is an ideal:  this is true if either is 0.  If not, we can
use (3) to characterize the $J$-special integral closure.
If $r_1$ and $r_2$ satisfy condition (3) with integers $n_1$ and $n_2$,  their sum satisfies
condition (3) with  $n = \max\{n_1,\,n_2\}$, while closure under multiplication is obvious.
In the general case, the $J$-special integral closure is the intersection of the images of
what one gets modulo various minimal primes.

The second statement in (2) shows that $J$-special integral closure is contained in $J^t$-special
integral closure, and the opposite inclusion is obvious.  It is clear that the special part of the integral
closure contains $JI$,  and we may replace $J$ by $\Rad(J)$. It is also obvious
that $\spi IJ$ is contained in both $\Rad(J)$ and $\oI$. 
\end{proof}

\begin{cor}
Let  $(R,\,\m)$ be a local domain.  An element $r\in R$ is in the special part of the integral closure
of $I$ iff for all $R \inj V$ with  $V$ a DVR centered on $\m$, if $IV \not=0$
then $\ord_V(r) > \ord_V(I)$.
\end{cor}

\begin{proof}
Since $\ord_V(\m)$ will be positive, the condition is necessary.  Because there are
only finitely many valuation rings needed in the test in part (5) of Theorem~\ref{thm:testcl}, we can always choose
a value of $n$ that will work for all of these. 
\end{proof}

The  {\it inner integral closure} of $I$ is defined as the $I$-special integral closure of $I$.
Instead of $\spi I I$ we write $\inn I$.  This construction is developed to some extent 
in~\cite[Section 10.5]{HuSw-book} (from which
we obtain our notation) and in \cite[Section 4]{GaVit-wsi}; some of our results in this section
may overlap with the results in these two references.  We emphasize that this {\it is not a closure
operation}.  It does not usually contain $I$,  but is contained in the integral
closure of $I$.  It may be thought of as the ``inner" part of the integral 
closure. Note that by part (c) of Proposition~\ref{pr:intbasic} below, the inner
integral closure of $I$ is the same as the inner integral closure of $\ici{I}$, and by part
(d) of Proposition~\ref{pr:intbasic}, the ideal $\inn I$ is itself integrally closed.  The notation 
$\inn I$ agrees with the notation given in \cite[Definition 10.5.3 on p.\ 206]{HuSw-book},  where 
an infinite family of related notions ($I_{>\alpha}$ for all positive rational numbers $\alpha$) is 
considered, but no name is given for $\inn I$.  

\begin{thm}\label{thm:int}
Let $R$ be a Noetherian ring, $I$ an ideal of $R$, and
$r  \in R$. The following conditions are equivalent, and characterize when $r$ is in 
the inner integral closure of $I$.  
\begin{enumerate}
\item The element $r$ satisfies a monic polynomial  $f(x) \in R[x]$ of some degree
$d$ such that the coefficient of $x^{d-i}$ is in $I^{j+1}$,  $1 \leq j \leq d$.
\item There is an integer $n$ such that $r^n$ is integral over
$I^{n+1}$.
\item There is an integer $n$ such that $r^n \in I^{n+1}$.
\item  For every prime ideal $P$ containing $I$,  $r/1$ is in the
special part of the integral closure of $IR_P$.  
\item For every map $R \to V$,  where $V$ is a DVR, such that
$IV$ is not 0 and not $V$,   $\ord_V(r) > \ord_V(I)$.
\item The condition in (5) holds for all $V$ such that $R \to V$
kills a minimal prime $\p$ of $R$.
\item  If $R$ is a domain, and $I = (\vect f h)$ where the $f_i$ are nonzero, 
the condition in (5) holds for all  $V$ arising as the localization at a minimal
prime of one of the $f_i$ in the normalization of one of the rings $R[I/f_i]$.  
\item If $R$ is an excellent domain and every prime ideal is an intersection of
maximal ideals\footnote{\emph{i.e.} $R$ is a Hilbert ring.} (e.g., if $R$ is a finitely generated $\C$-algebra),
the condition in (5) holds for discrete valuations centered on a maximal
ideal $m$ of $R$.
\end{enumerate}
\end{thm}

\begin{proof}
We already have $(1) \imp (2) \imp (5) \imp (6) \imp (7)$.  We can see that
$(2) \dbimp (3)$ as follows. It suffices to show $\imp$:  the converse is
obvious.  But if $r^n$ is integral over $I^{n+1}$ then there exists $h$ such
that for all $N$,  $(I^{n+1} + r^n)^{N+h} = (I^{n+1})^{N+1} (I^{n+1} + r^n)^{h-1}$,
so that $r^{n(N+h)} \in I^{(n+1)(N+1)}$ for all $N$.  But if $N >  nh-h -1$ we have
that $(n+1)(N+1) > n(N+h)$, as required.

 Because the condition
in (7) (and in part (5) of the preceding theorem) involves only finitely many DVRs, we 
can choose $n$ that works for these, and then the preceding theorem yields (1) and (2).
Condition (4) and the preceding Corollary yield condition (5),  since $IV \not= V$ in (5) 
implies that $V$ is centered on a prime containing $I$.  

It remains only to prove (7), which is obviously necessary.  When $R$ is excellent,
the normalization $S_i$ of $R[I/f_i]$ is finitely generated over $R$.  Let $V$ be one
of the localizations of $S_i$ at a minimal prime of $f_j$.  Let $v$ generate the
maximal ideal of $V$.  We can choose a localization $T_0$ of $S_i$ at one element
such that $v$ generates a prime ideal in $T_0$.  We can write  $r$ and $f_i$ as units
of $V$ times powers of $v$,  say $r = uv^\nu$,  $f_i = u'v^\mu$, where $\nu > \mu$.  
We may localize $T_0$ further, but at finitely many elements, so that it contains  $u$ and $u'$.  
Call the resulting ring $T$.  Then $v$ is prime in $T$, and the domain $T/vT$ is excellent.
It follows that there is a maximal ideal  $\m$ of $T/vT$  such that $(T/vT)_{\cM}$ is regular.
The pullback  $\cM$ of $\m$ to $T$ will be a maximal ideal of $T$ containing $v$ such
that $T_\cM$ is regular and the image of $v$ is a regular parameter.  The contraction
of $\cM$ to  $R$ is a maximal ideal because $R$ is a Hilbert ring. The valuation given
by order with respect to powers of $\cM$ will give the same result for the orders
of $r$ and $I$ as $V$.  
\end{proof}

\begin{prop}\label{pr:intbasic}
Let $I$ be an ideal of the Noetherian ring $R$.
\begin{enumerate}[label=(\alph*)]
\item  If  $W$ is a multiplicative system of $R$,   $\inn{(I_W)}= (\inn I)_W$.
\item  If  $S$ is integral over $R$,  $\inn{(IS)} \cap R = \inn I$.
\item  If $I \inc J$ then $\inn I \inc \inn J$.  If $I \inc J \inc \ici{I}$, then
$\inn I = \inn J$.  
\item The ideal $\inn I$ is integrally closed.
\end{enumerate}
\end{prop}

\begin{proof}
For part (a),  if $r^n \in I^{n+1}R_W$,  then for some $w \in W$ we have
$wr^n \in  I^{n+1}$,  and then $(wr)^n \in I^{n+1}$, which shows  that
$wr \in \inn I$.  The other implication is trivial.  

For (b),  if  $r \in R$ is such that $r^n$ is in the integral closure of $I^{n+1}S$,
then $r^n$ is in the integral closure of $I^{n+1}$.  The other containment is
obvious.

The first statement in (c) is obvious from condition (2) or (3) of Theorem~\ref{thm:int}.
The second statement is clear provided that we can show that an element $r$
in the inner integral closure of $\ici{I}$ is in the inner integral closure
of $I$.  But  $(\ici I)^{n+1} \inc \icip{I^{n+1}}$,  and so if $r^n$ is in the
former it is in the latter.  

Part (d) follows because condition (5) of Theorem~\ref{thm:int} shows that $\inn I$
is an intersection of valuation ideals.  This fact also follows from 
\cite[Proposition 10.5.2 part (4)]{HuSw-book} coupled with the final remark in the paragraph 
following \cite[Definition 10.5.3]{HuSw-book}.
\end{proof}

\begin{thm}\label{thm:intcont}
Let $R$ be a finitely generated $\C$-algebra.  Then $\inn I \subseteq I\ct$.
\end{thm}

\begin{proof}
Consider a polynomial for $r$ as in condition (1) of Theorem~\ref{thm:int}, with indeterminate 
variable $x$.  Suppose that the coefficient of $x^{d-j}$
is a sum of $t_j$ terms, each the product of $j+1$ elements of $I$.  We replace all $t_j(j+1)$
elements of $I$ involved by variables.  This gives an equation $F =
x^d + P_1 x^d_1 + \cdots P_d = 0$  where $P_j$ is homogeneous of degree $j+1$ in
$t_j(j+1)$ variables, and all the variables from distinct  $P_i$, $P_j$  are mutually disjoint.
Consider the ring $S$ defined by adjoining $x$ and all the other variables to $\C$  and killing
the polynomial $F$.  In this ring, let $J$ be the ideal generated by all variables other
than $x$.  The radical of $J$ contains $x$ as well.   We have a homomorphism $S \to R$
which takes $x$ to $r$ and such that $JR \inc I$.  Thus, it suffices to show that
the image of $x$ is in $J\ct$ in $S$.  But $S$ has an $\N$-grading:  we give the variables occurring
in $P_j$ degree $(d+1)!/(j+1)$, $1 \leq j \leq d$,  and we give $x$ degree $(d+1)!$.  Since
$x$ has a higher degree than any generator of $J$, the result follows from Theorem~\ref{thm:highdegct}
\end{proof}

It follows that $\inn I \subseteq I\ax$ in an affine $\C$-algebra.  In fact this holds more generally.  To see this, first, we give the following useful lemma.

\begin{lemma}\label{lem:mJ}
Let $(R,\m,k)$ be a 1-dimensional complete Noetherian seminormal local ring.  Then for any ideal $J$ of $R$, we have $\m \cdot \ici{J} \subseteq J$.
\end{lemma} 

\begin{proof}
First, note that it is enough to prove the lemma for a \emph{primary} ideal.  For let $J = Q_1 \cap \cdots \cap Q_t$ be a primary decomposition, and suppose the lemma holds for primary ideals.  Then $\m \ici{J} \subseteq \m \ici{Q_1} \cap \cdots \cap \m \ici{Q_t} \subseteq Q_1 \cap \cdots \cap Q_t = J$.

Let $L_i$ and $V_i$ be as in the description given in Theorem~\ref{thm:glue}.  For each $i$, let $t_i$ be a generator of the maximal ideal of $V_i$.  We have $\Spec R = \{\p_1, \dotsc, \p_n, \m\}$, where $\p_i$ is the kernel of the map $R \rightarrow V_i$.

If $J$ is $\p_i$-primary for some $i$, it follows from the minimality of $\p_i$ and the fact that $R$ is reduced that $J=\p_i$.  Then $\m \ici{J} = \m \ici{\p_i} = \m \p_i \subseteq \p_i = J$.

Thus, we may assume $J$ is $\m$-primary.  Then for each $i$, there is some integer $1\leq e_i<\infty$ such that $J V_i = t_i^{e_i} V_i$.  Hence, \[
\ici{J} = \bigoplus_{i=1}^n t_i^{e_i} V_i.
\]
Pick any $i$ between 1 and $n$.  Then from the structure of $R$, it follows that there is some element of the form $c:=u t_i^{e_i} + \sum_{j\neq i} v_j t_j^{f_j} \in J$, where $u$ is a unit of $V$, the $v_j \in V$, and the $f_j \geq e_j$.  Then for a typical element $v t_i^{e_i+1} \in t_i^{e_i+1}V_i$ (where $v\in V_i$), we have  $u^{-1} v t_i \in R$, so that $v t_i^{e_i + 1} = (u^{-1} v t_i) c_i \in J$.  Hence, $\m \ici{J} = \bigoplus_{i=1}^n t_i^{e_i+1} V_i \subseteq J$.
\end{proof}

Next, we have the following theorem, which follows from Theorem~\ref{thm:intcont} when $R$ is a finitely generated $\C$-algebra, but in fact holds in a more general setting as we see below.

\begin{thm}\label{thm:intax}
Let $R$ be an excellent Noetherian ring, and $I$ an ideal.  Then $\inn I \subseteq I\ax$.
\end{thm} 

\begin{proof}
Since the persistence property holds for both inner integral closure and axes closure, we may assume that $(R,\m,k)$ is a complete local 1-dimensional seminormal ring, in which case what we want to show is that $\inn I \subseteq I$.  Pick Noetherian valuation rings $(V_i, (t_i), L_i)$, $1\leq i\leq n$, as in Theorem~\ref{thm:glue}.  If $I=R$ there is nothing to prove, so we may assume $I \subseteq \m = \bigoplus_{i=1}^n t_i V_i$.  For each $i$, either $I V_i = 0$ or $I V_i = t_i^{e_i} V_i$ for some $1\leq e_i < \infty$.  If we use the convention $t_i^\infty=0$, then we have $\ici I = \bigoplus_{i=1}^n t_i^{e_i} V_i$, where $1\leq e_i \leq \infty$ for each $i$.

By part (5) of Theorem~\ref{thm:int}, we have $\inn I \subseteq \bigoplus_{i=1}^n t_i^{e_i+1} V_i = \m \ici{I}$.  But by Lemma~\ref{lem:mJ}, we have $\m \ici{I} \subseteq I$, which completes the proof.
\end{proof}

We next observe that $\inn I = \inn{(\oI)}$, and so the operation that sends
$I$ to $I + \inn I$ is a closure operation in the sense of Definition~\ref{def:closure}.

\begin{defn}\label{def:natrelevant}
For an ideal $I$, we let $I\ncl := I + \inn I$, the \emph{natural closure} of $I$.  If $I =I\ncl$
we say that $I$ is {\it naturally closed}.  Evidently, if  $I \inc I_1 \inc \oI$ and $I  = I\ncl$,
then $I_1 = I_1\ncl$.

By a {\it valuation} of $R$ we simply mean a map $R \to V$ where $R$ is a discrete valuation ring.
It may have a kernel, even if $R$ is a domain. If $IV \not= 0$ has order $k > 0$ call the valuation ideal 
$\fA$ arising from
contraction of $\n^{k+1}$ the {\it $I$-\rv\ valuation ideal}  of $R \to V$.  If $IV = 0$, we call
the contraction of $0$, i.e., $\Ker(R \to V)$, the {\it $I$-\rv\ valuation ideal} of $R \to V$.
\end{defn}  

If  $R$ is a domain and $I$ is a nonzero ideal of $R$,  then for every nonzero
element $f \in I$ we may take the normalization $S$ of the $R[I/f]$  and consider
the discrete valuation rings arising as localizations of $S$ at a minimal prime of $fS$.
There are only finitely many such valuation rings, and if $R$ is excellent, they
are small.  Following standard terminology, we refer to these rings as the {\it Rees} valuation rings of $I$.

In the case where $R$ is not a domain, we make the following conventions
about the Rees valuations of $I$.  For every minimal prime $P$ of $R$ not containing
$I$, we include the Rees valuations of $I(R/P)$ among the Rees valuations
for $I$.  If  $P$ contains $I$, we shall think of the fraction field $\kappa(P)$ of
$R/P$ as a ``degenerate"  valuation ring, and we include the maps $R \to \kappa(P)$
as Rees valuations.  The order of every element not in $P$ is 0,  while the
elements of $P$ may be viewed as having order $+\infty$.   

\begin{thm} 
Let $R$ be Noetherian and $I$ and ideal of $R$.
\begin{enumerate}[label=(\alph*)]
\item  There are finitely many valuations $R \to (V_i, \n_i)$ such that $\inn I$ is the intersection
of the $I$-\rv\ valuation ideals of these valuations.  The $V_i$ may be chosen to be the
Rees valuations.  

\item $I = I\ncl$ if and only if there are finitely many valuations $R \to (V_i, \n_i)$ with $I$-\rv\ valuation
ideals $\fA_i$ such that $\cap_i \fA_i \inc I$.   These valuations may be chosen to be the Rees
valuations of $I$.
\end{enumerate}
\end{thm}

\begin{proof}
For part (a),  we know that $r \in \inn I$ if and only if that holds modulo each minimal prime 
of $I$. Thus, it suffices to show that there are valuation ideals as specified for every $R/\p$ and $IR/\p$.
We may thus reduce to the domain case.  If $I \inc \p$ we take the \rv\ valuation ideal to be 0.  We may 
localize at any height one prime of the normalization of $R/\p$ and do this.  Otherwise we use the
valuation ideals coming from the rings $R[I/f]$. 

The ``only if'' part of (b) is obvious.  For the ``if"  part, suppose that we have $\cap_i \fA_i \inc I$.
To show that $I = I\ncl$, it suffices to show that $\inn I \inc I$.  Let $u \in I$.  Then $u \in \fA_i$
for all $i$,  and the result is obvious.
\end{proof}

\begin{prop}\label{pr:intloc}
Let $I$ be an ideal of $R$ and  $W$ a multiplicative system in
$R$.  Then $(I_W)\ncl = (I\ncl)_W$.
\end{prop}
\begin{proof}
 $(I_W)\ncl = I_W + \inn{(I_W)} = I_W + (\inn I)_W$,  by Proposition~\ref{pr:intbasic}(a), but 
$I_W + (\inn I)_W = (I + \inn I)_W = (I\ncl)_W$. 
\end{proof}

\begin{thm}\label{thm:RtoR(t)}
Let $R \to S$ be a flat homomorphism of excellent Noetherian
rings whose fibers are geometrically regular.  Let $I$ be ideal of $R$.  Then
$\inn{(IS)} = \inn I S$, and $(IS)\ncl = (I\ncl)S$.  
In particular,  if $\rmk$ is local, this holds if $S$ is $\wh{R}$
or $S = R(t)$, the localization of the polynomial ring $R[t]$ at $mR[t]$.
\end{thm}
\begin{proof}
The statement for natural closure is immediate from the statement for inner
integral closure.  Evidently, $\inn I S \inc \inn{(IS)}$.  To see the opposite inclusion,
it will suffice to show that for every $I$-\rv\ ideal $\fB$ for a Rees valuation
of $I(R/p)$, where $p$ is a minimal prime of $R$, $\fB S$ is an $I$-\rv\ ideal
for $IS$ of a valuation on $S$.  For $\inn I$ is the intersecton of the ideals $\fB$,   and the intersection
of the ideals $\fB S$ will be $(\inn I)S$, and will also contain $\inn{(IS)}$. 

To prove this, we may replace  $R \to S$ by $R/p \to S/pS$, which is still flat with
geometrically regular fibers.  Thus, we may assume that $R$ is a domain.  
Let $(V,\, \n)$ be a Rees valuation of $I$.  Then $V$ is essentially of finite type over $R$,  since $R$ is excellent, 
and so $T = S \otimes_R V$ is essentially of finite type over $S$, and is Noetherian
and flat over $V$ with geometrically regular fibers.  It follows that $T/\n T$
is reduced.   Let $\vect Q s$ be the minimal primes of $\n T$.  Then every $T_{Q_i}$ is a valuation
ring, and $Q_iT_{Q_i} = \n T_{Q_i}$.   Let $h$ be the order of $IV$, so that $IV = \n ^h$. Then
$(IS)T_{Q_i} =(IV)T_{Q_i} = \n^hT_{Q_i} = Q_i^hT_{Q_i}$, for each $i=1, \dotsc, s$.  Hence, the contraction
$\fA_i$ of  $Q_i^{h+1}$  to  $S$ is an $IS$-relevant ideal for  the valuation ring $T_{Q_i}$. 
Consequently,  $\inn{(IS)}  \inc \fA_i$. 

We next show that $\bigcap_i \fA_i = \fB S$.  Since $\n^{h+1}$ maps into $Q_i^{h+1}T_{Q_i}$,  
$\supseteq$ is clear.  We need to show that $\bigcap_i \fA_i \inc \fB S$.  First observe
that since  $(R/\fB) \inj V/\n^{h+1}$ is injective and $S$ is $R$-flat,  we may apply
$S \otimes_R \blank$ to obtain that  $$S/\fB S \inj  (V/\n^{h+1}) \otimes_R S \cong T/\n^{h+1}T$$
is injective.  Thus, $\n^{h+1}T$ lies over  $\fB S$.  Since $V$ is a discrete valuation ring and
$V \to T$ is flat with regular fibers, $T$ is regular.  (In fact, we only need that $T$ is normal.)
Moreover, $\n$ is principal and so $\n^{h+1}T$ is a principal ideal generated by a non-zerodivisor.  It
follows that its primary decomposition is simply  $\bigcap_{i=1}^s Q_i^{(h+1)}$,  since
$Q_iT_{Q_i} = \n T_{Q_i}$.  Hence,  $\fB S$,  which is the contraction of $\n^{h+1}T$,  is
also the contraction of $\bigcap_{i=1}^s Q_i^{(h+1)}$.  This is the same as the intersection
of the contractions of the ideals  $Q_i^{(h+1)}$.   But  $Q_i^{(h+1)}$ is simply the contraction
of $Q_i^{h+1}T_{Q_i}$ to $T$,  and it follows that the contraction of $Q_i^{(h+1)}$ is  $\fA_i$. 
Thus,  $\bigcap_{i=1}^s \fA_i = \fB S$,  as claimed.
 
Finally, since $\inn{(IS)} \inc \fA_i$ for every $i$,  we have that $\inn{(IS)} \inc  \fB S$ for each of the $\fB$.  
Since there are only finitely many $\fB$, finite intersection commutes with flat base change, and the intersection 
of the $\fB$ is $\inn I$, we then have   $\inn{(IS)} \inc \inn I S$, proving the other needed inclusion.  \qed
\end{proof}

It follows from the proof that the hypothesis that the fibers be geometrically regular may be weakened:  it suffices  that the
fibers are geometrically reduced and whenever $R \to V$ is a map to a discrete valuation ring,
$V \otimes_R S$ is normal.  

In the case of a finitely presented smooth $R$-algebra $S$, we do not need any hypothesis of excellence on $S$.

\begin{prop}\label{pr:natsmooth} Let $R$ be a ring and let $S$ be a finitely presented $R$-algebra that is smooth over $R$.  Then for every
ideal $I$ of $R$,  $(IS)_{>1} = I_{>1}S$ and $(IS)\ncl = I\ncl S$. \end{prop}
\begin{proof} The second conclusion follows from the first, and, for the first, it suffices to prove that
$(IS)_{>1} \inc I_{>1}S$.  First note that there is a subring  $R_0$ of $R$ finitely generated over the prime
ring $\Lambda$ in $R$ and a finitely presented smooth $R_0$-algebra $S_0$ such that $S \cong R \otimes_{R_0} S_0$, and
that $S$ is the union of the rings  $S_1 = R_1 \otimes_{R_0} S_0$  as $R_1$ runs through all subrings of $R$ finitely generated
over $R_0$:  these in turn are finitely generated over $\Lambda$ and, therefore, excellent.   
 If $f \in (IS)_{>1}$  then we can choose such an  $R_1$ so that $f \in (I_1S_1)_{>1}$,  where $I_1$ is an ideal of $R_1$ generated
 by elements in $I$.  But then $R_1$ and $S_1$ satisfy the hypotheses of Theorem~\ref{thm:RtoR(t)}, and so we can conclude
 that $f \in (I_1)_{>1}S_1 \inc I_{>1}S$, as required. \end{proof}

\begin{prop}\label{pr:nchomog} Let $R$ be a $\Z^h$-graded ring, $h>0$, and let $I$ be a homogeneous ideal with respect to this grading. Then
$I_{>1}$ and $I\ncl$ are also homogeneous with respect to this grading. \end{prop}
\begin{proof} It suffices to prove this for $I_{>1}$.  Although $_{>1}$ is not a closure operation, Discussion~\ref{disc:grad}  applies:  it
is defined ring-theoretically, and will be stable under the automorphisms $\theta_\alpha$.  The only issue in the proof is that
$R$ may not have sufficiently many units.  However, the rings $S_N$ introduced in Discussion~\ref{disc:grad}  are finitely presented,
smooth, and faithfully flat over $R$,  and the result now follows from Discussion~\ref{disc:grad}, Proposition~\ref{pr:natsmooth}, and
the fact that since $S_N$ is faithfully flat over $R$,  $I_{>1}S \cap R = I_{>1}$.\end{proof}

\section{The ideal generated by the partial derivatives}\label{sec:partial}
\begin{thm}\label{thm:Jac}
 Let  $S$ be a localization of $R = \C[\vect x n]$ at one element.  Let $f \in R$,
and let  $\cJ =( \partial f/\partial x_1, \, \ldots, \,\partial  f/\partial x_n)$. \begin{enumerate}[label=(\alph*)]
\item For any nonconstant $f \in R$ and integer $N$ there exists $g \in S$
such that $(1 - gf)f$ is in the inner integral closure of $\cJ$ and, hence, the continuous closure.
\item If $f \in \Rad(\cJ S)$ then $f$ is in the inner integral closure of $\cJ S$ and hence in the continuous
closure of $\cJ S$.
\end{enumerate}
\end{thm}

\begin{proof}
We first prove (b).  Consider a valuation $S \to V$ centered on a maximal ideal $\cM$ 
that contains $\cJ$.  Then we have $S_\cM \to V$, and $f$ is also in $\cM$.  We may assume 
that $\cM$ corresponds to the
origin, and so we have  $\theta:\C[\![\vect x n]\!] \to L[\![t]\!]$ for some field $L$ with $\C \inc L$.  
Then  $\theta(f)$ is in the maximal ideal of $L[\![t]\!]$ and is nonzero. By the chain rule, its derivative
is in $\cJ L[\![t]\!]$.  The derivative has order exactly one less than that  of $\theta(f)$.  This shows that
the order of $f$ is strictly larger than the order of $\cJ$.

To prove (a),  let  $\fB = \Rad(\cJ): f$.  Then no maximal ideal can contain $\fB + fR$:  when we localize
at such a maximal ideal,  the chain rule shows that for all valuations centered on it,  the order
of $f$ is larger than the order of $\cJ$,  and so $f \in \ici{\cJ} \inc \Rad(\cJ)$,  which contradicts the fact
that the maximal ideal contains $\Rad(\cJ):f$.  Thus, we can choose  $h \in \fB$  such
that that $h + fg = 1$,  and so $1 - fg \in \fB$,  and $f(1 - fg) \in \Rad(\cJ)$.  We can
apply part (b) to the ring $R_h$ where $h = 1 - fg$.  Hence,  $h$ has a power,  also of the
form $1 - fg$, which multiplies $f$  into the inner integral closure of $\cJ$. 
\end{proof}

\section{Naturally closed primary ideals are axes closed}\label{sec:natax}

In this section we prove that if $I$ is a primary ideal of
an excellent Noetherian ring,  then $I = I\ax$ iff $I = I\ncl$.  This shows
that for ideals primary to maximal ideals,  natural closure and axes closure
agree.  Moreover, in the case of an affine $\C$-algebra, since $I\ncl \inc I\ct \inc I\ax$,
we have the corresponding result for continuous closure as well.   It then follows that
for an unmixed ideal $I$ of an affine $\C$-algebras,  $I = I\ct$ if and only if $I = I\ax$. 

We first extend the $I$-\rv\ terminology from Definition~\ref{def:natrelevant} as follows.

\begin{defn}
Let  $D$ be a Dedekind domain with a 
nonzero principal ideal $tD$ where $t$ is prime. Given a map $R \to D$ we say that $\fA$ is 
the $I$-\rv\ associated with $(D, \, tD)$ if  either  $ID = 0$ and $\fA = \Ker(R \to D)$ or $ID = t^kD$ 
and $\fA$ is the contraction of $t^{k+1}D$.
\end{defn}  

This is the same as saying that $\fA$ is the 
$I$-\rv\ ideal of $R \to D_Q$,  where $Q = tD$.  

We need the following result, which is contained in \cite[Corollary 1, p.\ 158]{EisHo-Zar}.

\begin{thm}\label{thm:EisHo}
Let $R$ be a regular Noetherian ring,  let $P$ be a prime ideal of
$R$,  and let $\cM$ by a family of maximal ideals of $R$ containing $P$ whose intersection
is $P$,  and such that for all $\m \in \cM$,  $R_\m/PR_\m$ is regular.  Let
$n$ be a positive integer.  Then $\bigcap_{\m \in \cM} \m^n = P^{(n)}$, the
$n\,$th symbolic power of $P$. 
\end{thm}

\begin{cor}\label{cor:EisHo}
 Let $R$, $P$ and $\cM$ be as in Theorem~\ref{thm:EisHo}.  Moreover,
suppose that the prime ideal $P$ is prinicipal with generator $\pi$.  Let $\vect n k$
be finitely positive integers.  Then the set of maximal ideals $\m$ in $\cM$
such that $\pi^{n_i} \in \m^{n_i} - \m^{n_{i+1}}$ also has intersection $P$.   In particular,
this set is non-empty. 
\end{cor}

\begin{proof}
By a straightforward induction on $k$,  we reduce to the case where $k = 1$.
We write $n$ for $n_1$.  Let $f \in R-P$.  It suffices to show that there exists $\m \in \cM$
such that $f \notin \m$ and $\pi^n \notin \m^{n+1}$.  Let $g \in R- P$.  Then there is
an element of $\cM$ that does not contain $g$.  It follows that the intersection of the set $\cN$
of maximal ideals in $\cM$ that do not contain $f$ is also $P$.  By Theorem~\ref{thm:EisHo}, 
$\bigcap_{\m \in \cN} \m^{n+1} = P^{(n+1)} = \pi^{n+1}R$.  Therefore, we can choose
$\m \in \cN$ such that $\pi^n \notin \m^{n+1}$. 
\end{proof}

\begin{keylem}
Let $\rmk$ be an excellent local domain of dimension 
$d \geq 2$ with infinite residue class field and let $I$ be an $\m$-primary ideal. 
Let $f$ be a nonzero element of $I$.   Let $S$ be the normalization 
of $R[I/f]$, and let $R \to V = S_P$ be a Rees valuation of $I$, where $P$ is 
a minimal prime of $fS$.  Let $\n$ denote the maximal ideal of $V$. 
Let $IV$ have order $h-1 \geq 1$, and let  $\fB$ be the 
contraction of the  proper nonzero ideal $\n^h$ of $V$ to $R$.  Let $\vect g h \in R-\{0\}$,  and let
$n_i = \ord_V(g_i)$.   Then there  exists  an algebra $T$ finitely generated over $S$ 
by adjunction of fractions, a prime  ideal $Q$ of $T$ of height $d-1$ such that $T/Q$ 
is a Dedekind domain, and a principal  height one prime $\pi T/Q$ of $T/Q$ such 
that $I T/Q = \pi^{h-1}T/Q$,  $\fB T/Q =  \pi^h (T/Q)$,  and the image of $g_i$ in $T/Q$ is 
a unit times $\pi^n$.  In particular, by including a given nonzero element $r$ of $R$ among the $g_i$, we
may choose $Q$ so that $r$ has nonzero image in $T/Q$.  
\end{keylem}

\begin{proof}
By the dimension formula, $S/P$ is an affine $K$-algebra of dimension $d-1$.
Since $S_P$ is regular, we can localizate at one fraction of $S-P$ to produce
$S_1$ that is regular, and we can localize at one element of $S-P$ to produce
a localization $S_2$ such that $S_2/PS_2$ is regular as well.  We may also localize
so that $P = yS_2$ is principal.  We replace $S$ by $S_2$ and drop the subscript. 
By Corollary~\ref{cor:EisHo} we can choose a maximal ideal $\m$ of $S_2$ containing $P$ so 
that the orders of a finite set of generators of $I$, and of $\fB$, as well as of $g$ are 
all the same with respect to the $m$-adic valuation on $S_\m$ as they are in $V$.  
We may extend $y$ to a set of elements $y_1 = y,\, y_2, \, \ldots, \, y_d$ in $\m$
that generate  $\m S_\m$.  We may localize at one element of $S-\m$ and then
assume that these elements generate $\m$.  Consider the leading forms of
$g_i$ and the various generators of $I$ and $\fB$ in $\gr_m S$.  After a linear
change of generators over $K$ (which is infinite) that fixes $y$ and
replaces each $y_i$ for $i >1$ by $y_i + c_iy_i$,  we may assume that all of
these leading forms have a term that is a scalar times a power of $y$:  the
exponent on $y$ is the order.  The ideal $S_\m/(y_2, \, \ldots, y_d)S_\m$ is
a regular domain in which the image of $y$ generates a prime ideal.  This will
remain true after we localize at one element of $S-\m$.  This localization is $T$,
and we may take $Q = (y_2, \, \ldots, \, y_d)T$.  
\end{proof}

\begin{cor}\label{cor:surjDed}
Let $R$ be an excellent ring, and let $I$ be an 
$\m$-primary ideal of $R$,  where $\m$ is maximal and $R/\m$ is infinite. 
$R \to V$ be a Rees valuation with $I$-\rv\ ideal $\fA$.  Let $r$ be a non-zerodivisor in  $R$.  
Then there exist finitely many surjections
$R \to D_i$ such that $D_i$ is a one-dimensional Dedekind domain that is finitely generated
over the residue class field $K = R/\m$
and nonzero primes $t_iD_i$ such that $\fA$ is the intersection of the $I$-\rv\ ideals $\fB_i$ of
these maps.  Moreover,  the surjections may be chosen in such a way that the image of $r$
in every $D_i$ is nonzero. 
\end{cor}

\begin{proof}
 If the Rees valuation has kernel $\p$ we may work with $R/\p$,  $IR/\p$ and $\fA/\p$.  The image
 of $r$ in  $R/\p$ is not 0.   Thus,
we may assume that $R$ is a domain, and that $r$ is a nonzero element of  $R$. 
Consider all finite intersections of ideals $\fB$ containing
$\fA$ of the type described (including the condition that the image of $r$ be nonzero in every $D_i$).  
Since $R/\fA$ has DCC, one of these is minimum.  If it is not $\fA$,
choose $g$ in it that is not in $\fA$.  By the Key Lemma,
we can construct $\fB$ so that the image of  $r$ in the corresponding Dedekind domain is not 0 and so that
it contains $\fA$ but not $g$, a contradiction.
\end{proof}

\begin{cor}\label{cor:reducetodim1}
Let $I$ be an $\m$-primary naturally closed ideal of an excellent ring 
 $R$, where $\m$ is maximal and $R/\m$ is infinite.   Let $r \in R$ be a
 non-zerodivisor.  
 Then there is a radical ideal $J \inc I$ such that $R/J$ has pure 
 dimension one,  the image of  $r$ is not a zerodivisor in $R/J$,  and $I(R/J) = I/J$ is naturally closed.  Moreover,
 $I/J$ is primary to $\m/J$ in $R/J$.  
 \end{cor}
 
\begin{proof}
Pick $\fA_i$ that are $I$-\rv\ from Rees valuation rings and whose intersection is contained
 in $I$.  For each $\fA_i$ pick $\fB_{ij}$ as in Corollary~\ref{cor:surjDed} whose intersection is within $\fA_i$,
 and let $\fq_{i}$ be the kernel of the map onto a Dedekind domain of dimension one
 that is used in construction $\fB_{ij}$.  This may be done so that  $r$ is not in any of $\fq_i$. 
 Take $J = \cap_{i,j} \fq_{ij}$.  Then $R/J$ has the required property.
 \end{proof}
 
 \begin{lemma}\label{lem:dimone}
 Let $R$ be a one-dimensional excellent Noetherian reduced ring, and let $S$ be the seminormalization
 of $R$. Let $I$ be an ideal of $R$ whose minimal primes are all height one 
 maximal ideals and such that $I$ is naturally closed.  
 Then $IS \cap R = I$. 
 \end{lemma}
 
 \begin{proof}
 Suppose that $f \in (IS \cap R) - I$.  
 By replacing $f$ by a multiple we may assume that $I:_R Rf$ is a maximal ideal $\m$ of $R$.
 If we replace $R$ by $R_\m$ and $S$ by $S_\m$ we still have that $f \notin IR_\m$,
 while natural closure commutes with localization.  It follows that we may assume that
 $(R,\,\m,\,K)$ is local  of dimension one.   Consider the local extension rings  $R_1$ of $R$
 with $R \inc R_1 \inc S$ and choose $R_1$ maximal in this family such that
 $f$ is not in the natural closure if $IR_1$, which will still be primary to the maximal
 ideal of $R_1$.  Thus, we may replace $R$ by $R_1$,  and $I$ by the natural
 closure of $IR_1$,  and we still have a counterexample.  If $R_1 = S$ we are done.
 Hence, there is some element $u \in S - R$ such that $u^2, \, u^3 \in R$.  
 Note that $R[u] = R + Ru$.  By replacing $u$ by a multiple we may assume that the 
 annihilator of the image of $u$ in $(R + Ru)/R$
 is a prime ideal of $R$.  Since $u$ is in the total quotient ring of $R$, it is multiplied into
 $R$ by a non-zerodivisor, and it follows that we may assume that the annihilator 
 is a height one maximal ideal, which must be $\m$.  Since $u^3 = (u^2)u \in R$,
 we must have $u^2 \in \m$,  and it follows likewise that $u^3 \in \m$.  
 
 Hence,  $R[u] = R+Ru$ is local with maximal ideal $\m + Ru$.   Let $J = IR[u]
 = I + Iu$.  We shall show that $Iu \inc \inn I \inc I$, since $I = I\ncl$
 First note that since $u^2 \in \m$,
 some power of $u$ is in $I$,  say $u^k \in I$.  Then $(Iu)^k = I^ku^k\inc I^{k+1}$,
 as required.  Since every element of $Iu$ is in $R$,  this shows that $Iu \inc \inn I$.
 Hence, $J = I$.  Now suppose  $f \in J + \inn J$,  where the calculation of $\inn J$
 is in $R[u]$.  Then $f = v + w$  where $v \in J = I$.  Then $f-v \in \inn IR[u] \cap R$.
 By (\ref{pr:intbasic}b),  $f-v \in \inn I$, and so $f \in I\ncl$,  a contradiction.
 \end{proof}
 
 \begin{thm}\label{thm:prnatax}
 Let $R$ be an excellent Noetherian ring.  Let $I$ be an ideal
 primary to a prime ideal $P$ such that $I$ is naturally closed.  Then $I$ is axes closed.  Moreover,
 if  $r \in R$ is not a zerodivisor, and  $f  \notin I\ax$,  then there is a map $R \to A$,
 where $A$ is a one-dimensional excellent seminormal ring, which may further be assumed
 to be complete local, such that $f \notin IA$
 and  the image of $r$ is a non-zerodivisor in $A$.
 \end{thm}

 \begin{proof}
 Let $f \in I\ax - I\ncl$.  Since natural closure commutes with localization
 by Proposition~\ref{pr:intloc},  we may replace $R$ by $R_P$:  $f$ is not in $IR_P$ since $I$ is primary
 to $P$.  By the persistence of axes closure, $f \in (I   R_P)\ax$.
 We have therefore constructed a new counterexample in which $P$
 is maximal.  We revert to our original notation and call the ring $R$, but we assume that
 $P = \m$ is a maximal ideal.  Second, we may replace $R$ by $R(t)$,  by Theorem~\ref{thm:RtoR(t)}.
 Hence, we may assume without loss of generality that the residue field $K$ is infinite.  The element $r$ is still a non-zerodivisor. 
 
 By Corollary~\ref{cor:reducetodim1},  we can preserve the fact that $f$ is in the axes closure
 but not the natural closure of $I$ while passing to a reduced local ring of pure
 dimension one that is a homomorphic image of $R$, and such that the image of $r$ is not a zerodivisor
 in this ring.   Thus, we need only consider 
 the issue in dimension one.  Let $S$ be the seminormalization
 of $R$.  Then $IS \cap R = I$ by Lemma \ref{lem:dimone}, and so $f \notin IS$, which shows that 
 $f \notin I\ax$ by Definition~\ref{def:axsemi}.  Note also that $r$ remains a
 non-zerodivisor in $S$.   Finally, we may replace $S$ by a suitable completed localization.
 \end{proof}
 
 We also note:
 
 \begin{thm}\label{thm:intnc} Let  $R$ be an excellent Noetherian ring.  Then the following conditions
 on an ideal $I$ are equivalent:
 \begin{enumerate}
 \item $I$ is axes closed.
 \item $I$ is an intersection of primary naturally closed ideals.
 \end{enumerate}
 \end{thm} 
 \begin{proof}  By Proposition~\ref{pr:intprimax}, we already know that $I$ is axes closed if and only if it is an intersection
 of primary axes closed ideals,  and for primary ideals, being naturally closed coincides
 with being axes closed by Theorem~\ref{thm:prnatax}.  \end{proof}
 
\begin{disc}\label{intflat}  Although we have not been able to determine whether axes closure commutes with localization, we can show that
 it commutes with smooth base change in certain cases.  Before stating our results, we recall the notion of {\it intersecton flatness}
 from \cite{HHbase}, where it is introduced just before the statement of (7.18).  An $R$-module $S$ (typically, $S$ will be an 
$R$-algebra here) is {\it intersection-flat\/} or $\cap$-{\it flat\/} if $S$ is flat and for every finitely generated $R$-module $M$ and every collection of submodules $\{M_\lambda\}_{\lambda\in\Lambda}$ of $M$, the obvious injection
$$S\otimes_R(\bigcap_\lambda M_\lambda) \inj \bigcap_\lambda (S\otimes_R M_\lambda).$$ 
is an isomorphism.  The $\cap$-flat modules include  $R$ and are closed under arbitrary direct sum and passing to direct summands and therefore 
include the projective $R$-modules.  In \cite{HHbase} it is observed that
if $R$ is complete local and $R\to S$ is flat local then $S$ is $\cap$-flat over $R$, using Chevalley's theorem, and that $R[\![x]\!]$, where $R$ is 
Noetherian and $x$ denotes a finite string of variables, is $\cap$-flat over 
$R$.  

If $S$ is an $\cap$-flat $R$-algebra that is faithfully flat, where  $R$ is Noetherian, and $W$ is a multiplicative system consisting
of elements of $S$ that are nonzerodivisors on $S/PS$ for every prime ideal $P$ of $R$ (which implies that no element of
$W$ is in $mS$ for $m$ maximal in $R$),  then by Lemma (5.10) of \cite{AHH}, 
$W^{-1}S$ is $\cap$-flat.
In particular, the localization of a polynomial ring over $R$ at a multiplicative system of polynomials each of which has the property that its
coefficients generate the unit ideal is $\cap$-flat over $R$.
 \end{disc}
 
 \begin{thm}\label{thm:axsmooth} Let $R \to S$ be a flat homomorphism of excellent Noetherian rings with geometrically
 regular fibers.   
 \begin{enumerate}[label=(\alph*)]
 \item  If $I$ is an unmixed ideal of $R$ that is axes closed, then $IS$ is axes closed.
 \item If, moreover, $S$ is $\cap$-flat, then  for every $I$ of $R$,  $(IS)\ax = I\ax S$. In particular, this holds 
 when $S$ is a localization of a polynomial ring in finitely many variables over an excellent ring $R$ at a 
 multiplicative system consisting of polynomials each which has the property that its coefficients 
 generate the unit ideal in $R$. 
 \end{enumerate}  
 \end{thm}
 \begin{proof}  For  part (a),  note that the primary components of $I$ are axes closed, and since finite intersection
 commutes with flat base change it will suffice to prove the result when $I$ is primary, say  to $P$.  Then $IS$ is naturally closed in 
 $S$ by Theorem~\ref{thm:RtoR(t)}.  If $IS$ is unmixed, its primary components will also be naturally closed and so
 axes closed, and it will follow that $IS$ is axes closed.  But $R/I$ has a finite filtration by torsion-free $R/P$-modules.
 It follows that the associated primes of $S/IS$ are the same as those of $S/PS$.  Since $R/P$ is a domain and
 $R/P \to S/PS$ is flat with geometrically regular fibers,  $S/PS$ is reduced.
 
For part (b), note that by Theorems~\ref{thm:intnc} and~\ref{thm:prnatax},   $I\ax$ is an intersection of axes closed primary
 ideals $J_\lambda$.  Then $I\ax S =  \bigcap_\lambda J_\lambda S$, and so it suffices to show that every $J_\lambda S$ is
 axes closed, which follows from part (a).   \end{proof}
 
 \begin{thm}\label{thm:nzdtest} Let $R$ be an excellent Noetherian ring, let $I$ be an ideal of $R$, and let $r \in R$
 be a non-zerodivisor.  In testing whether $f \in I\ax$ in $R$,  it suffices to consider maps $h:R \to A$
 where $A$ is one-dimensional, seminormal, and $h(r)$ is not a zerodivisor in $A$.  Moreover, $A$ may
 be chosen to be complete local.
 \end{thm}  
 \begin{proof}  If $f \notin I\ax$,  by Proposition \ref{pr:intprimax}, we may choose a primary 
 axes closed ideal $J$  such that $I \inc J$ and $f \notin J$.  The result then follows from 
 Theorem \ref{thm:prnatax}. 
 \end{proof}  
 
 We do not know whether axes closure commutes with localization.  This would be true if an axes closed ideal 
 remained axes closed after localization.   The following
 result sheds light on the problem.
 
 \begin{thm}\label{thm:axesloc} Let $R$ be an excellent Noetherian ring,  and let 
 $I$ be an axes closed ideal. The following conditions are
 equivalent:
 \begin{enumerate}
 \item $I$ is a finite intersection of primary naturally closed ideals (and so has an irredundant primary
  decomposition in which the primary components are all naturally closed ideals). 
 \item $I$ is a finite intersection of primary axes closed ideals  (and so has an irredundant primary
  decomposition in which the primary components are all axes closed ideals).
 \item For every associated prime ideal $P$ of $I$,   $IR_P$  is axes closed in $R_P$.
 \item For every multiplicative system $W$ in $R$,  $IW^{-1}R$ is axes closed in $W^{-1}R$.
 \end{enumerate}
 \end{thm}  
 \begin{proof}  Note in (1) and (2) that if one can express $I$ as a finite intersection of naturally closed or 
 axes closed primary ideals, one may obtain an irredundant primary decomposition as usual by intersecting those primary to the same prime
 and omitting terms that are not needed. 
 
 The first two conditions are equivalent, since a primary ideal is naturally closed if and only
 if it is axes closed in an excellent ring by Theorem \ref{thm:prnatax}.   We have that $(1) \imp (4)$, 
 because primary naturally closed ideals remain  primary naturally closed ideals or become the 
 unit ideal when one localizes.  Since $(4) \imp (3)$ is
 evident, it suffices to show that $(3) \imp (2)$. Assume (3).  We use induction on the number of associated
 primes of $I$.  If there is only one associated prime, $I$ is primary to it,  and the result is clear.
  
We next reduce to the local case.  Let $P$  be an associated prime of $I$.  Then $IR_P$ is axes closed in $R_P$,
 by hypothesis, and if we can prove the result for $R_P$ and $IR_P$,  we may write $IR_P$  as a finite
 intersection of axes closed primary ideals in $R_P$.   The contractions of these ideals to $R$ will be
 finitely many axes closed primary ideals.  If we intersect all of these as $P$ varies, we obtain  $I$, for
 if  $f \notin I$  then $f$ has a multiple not in $I$ such that the annihilator of $(I+Rf)/I$ is an associated
 prime of $I$.   Thus,  $f/1 \notin  IR_P$,  and so it fails to be in at least one of the primary components $\fA$
 of $IR_P$, and $f$ will not be in the contraction of $\fA$ to $R$.  
  
 Hence, it suffices to prove the result when  $R = (R,\,P)$ is local and $P$ is an associated prime of $I$.  
 Let $J$ denote the ideal  $\bigcup_{t=1}^\infty (I:P^t)$, the saturation of $I$ with respect to $P$. Then $J$ is axes closed, by Proposition \ref{pr:axcolon}.  
 If we localize at any prime of $R$ other than $P$,  $I$ and $J$ have the same expansion.  Thus,  $I$ and $J$ have the same
 associated primes, except that $P$ is an associated prime of $I$ and not $J$,  and so $J$ has fewer associated
 primes than $I$.  Thus, $J$ remains axes closed when we localize at any of its associated primes.  
 By the induction hypothesis,  $J$ is a finite intersection of primary axes closed ideals.  We know that $J/I$  is killed by 
 a power of $P$, and so has finite length.  Let $\cS$ be the set of ideals contained in $J$ 
 that are finite intersections of primary axes closed 
 ideals that contain $I$.  Note that we have shown $J \in \cS$.  $\cS$  is a directed family
 by $\supseteq$,  since it is closed under finite intersection.  Since $J/I$ has finite length, we can choose
 $J_0 \in \cS$  such that the length of $J_0/I$ is minimum.  We can complete the proof by showing that $J_0 = I$.
 But if $f \in J_0 - I$, by Proposition \ref{pr:intprimax} we can choose a primary axes closed ideal $Q$ that contains $I$ and not $f$.
 Then  $Q \cap J_0 \in \cS$, and $Q/I$ is strictly contained in $J_0/I$ and, therefore, of smaller length, a contradiction.
 Thus, $J_0 = I$,  as required.  \end{proof}
 
 \begin{cor}\label{cor:prctax}
 For a primary ideal $I$ in an affine $\C$-algebra, $I = I\ax$,
  $I= I\ct$ and $I = I\ncl$ are equivalent.  Moreover, in an affine $\C$-algebra,  an unmixed ideal is
 continuously closed if and only if it is axes closed. 
 \end{cor}
 
 \begin{proof}
 We know that the continuous closure lies between the natural closure and
 the axes closure.  Hence, the first statement follows at once from Theorem~\ref{thm:prnatax}.
 If $I$ is unmixed and continuously closed, each primary component is a primary
 component for a minimal prime of $I$, and is continuously closed as well by
 Corollary~\ref{cor:ctcomponents}.  Hence, each primary component is axes closed, and an intersection
 of axes closed ideals is axes closed. 
 \end{proof}
 
 \begin{cor}\label{cor:zerodimctax}
 Let $I$ be an ideal of an excellent ring $R$ such that
 $R/I$ is zero-dimensional.  Then the natural closure of $I$ is the same as the
 axes closure of $I$.  Moreover, if $R$ is an affine $\C$-algebra, the continuous
 closure of $I$ is the same as well.
 \end{cor}
 
 \begin{proof}
 All ideals containing $I$ satisfy the same condition, and are therefore unmixed.
 Since an ideal containing $I$ is naturally closed if and only if it is axes closed,
 these two closures agree.  The final  statement follows from the fact that in an
 affine $\C$-algebra, $I\ncl \inc I\ct \inc I\ax$.
 \end{proof}  
 
\begin{prop}\label{pr:primmax}  If $R$ is a finitely generated $\C$-algebra,  $\m$ is a maximal ideal and 
$I$ is an $\m$-primary ideal,  then       
 $I\ct$ is the contraction of $(IR_\m)\ct$ to $R$.  This means that if $f \in R$,  then  $f \in I\ct$ if and only if  $f/1 \in (IR_\m)\ct$, 
 i.e.,  if and only if the germ of $f$ at the origin is in the expansion of $I$ to the ring of germs of continuous $\C$-valued 
 functions at the origin. \end{prop}
 \begin{proof} It is clear that $J = I \ct$ is contained in the contraction of $(IR_\m)\ct$ to $R$.  To complete the argument, it will
 suffice to show that $J$,  which is an $\m$-primary continuously closed ideal, is contracted
 from the ring $T$ of germs of continuous $\C$-valued functions at $x \in X = \MaxSpec(R)$ (with the Euclidean topology), where
 $x$ corresponds to $\m$. Because $J = J\ct$,  we know that $J = J\ax$ by Corollary~\ref{cor:zerodimctax} just above.  Let $f \in
 R - J$. Then by part (7) of Theorem~\ref{thm:axsemi}  we can choose  $\theta:R \to (A,\, \n)$ 
 such that $f \notin JA$,  where $\theta$ is a $\C$-homomorphism,
 $A$ is an analytic ring of axes over $\C$,  and $\theta^{-1}(\n) = \m$.  By Lemma 3.5 of \cite{Br-cc},  $JA$ is contracted is contracted
 from $T$.   Hence, $f \notin JT$.  Thus, $J$ is contracted from $T$, as required. \end{proof}  
 
\begin{thm}\label{thm:semiaxes}
Let $R$ be a reduced excellent ring.  $R$ is seminormal if and only if every principal ideal generated by
a non-zerodivisor is axes closed.
\end{thm}

\begin{proof}
We first prove ``if."  Suppose that $g$ is an element of the total quotient ring of $R$ such
that $g^2, \, g^3 \in R$.  We must show that $g \in R$.  Since $g$ is in
the total quotient ring of $R$, we can choose $f$, a non-zerodivisor in $R$,
such that $gf \in R$.  We claim that $gf$ is in the axes closure of $f$.  We use the test for being in the
axes closure provided by Theorem~\ref{thm:nzdtest}.  
Suppose that we have any ring map $\theta: R \to S$
such that $S$ is an excellent seminormal one-dimensional ring and $u:=\theta(f)$ is a non-zerodivisor.  Let $w := \theta(fg)$.  
Consider the element
$v=w/u$ of the total quotient ring of $S$.  We have \[
u^2 v^2 = (uv)^2 = w^2 = \theta(fg)^2 = \theta(f^2) \theta(g^2) = u^2 \theta(g^2).
\]
Since $u^2$ is a non-zerodivisor in the total quotient ring of $S$, it follows that $\theta(g^2) = v^2$, and in particular
that $v^2 \in S$.  Similar computations show that $\theta(g^3) = v^3$, so that $v^3 \in S$ as well.  Since $S$ is
seminormal, it follows that $v\in S$.  So we have \[
\theta(fg) = w = uv \in u S = \theta(f) S.
\]
Then by choice of $\theta$ and Theorem~\ref{thm:nzdtest}, it follows that $fg \in (fR)\ax = fR$.  That is, there is some $r\in R$
such that $fg = fr$.  Since $f$ is a non-zerodivisor, it follows that $g=r \in R$, so that $R$ is seminormal.

Now assume instead that $R$ is excellent seminormal.  We may assume
that $\dim R \geq 2$, since otherwise every ideal of $R$ is axes closed by Proposition~\ref{pr:onedimax}.  
Suppose some element $g \in R$ is in the axes closure of
$fR$, where $f$ is a non-zerodivisor.   Let $S$ be 
the integral closure of $R$.  Since $S$ is a product of finitely many normal domains (the normalizations of the
quotients of $R$ by its various minimal primes), every principal ideal
of $S$ is integrally closed, hence axes closed, so $g \in fS$.  That is, $g = hf$ for some $h\in S$. 
If we can show that $h$ is in the seminormalization of $R$ in $S$,  which is $R$, then we are done.  
To this end we use the criterion from Proposition~\ref{pr:semibasic}(1).
For the remainder of the proof we may change notations: we replace $R$ by $R_P$ (where $P$ is an arbitrary prime ideal of $R$),  and so
assume that $(R,\,P,\,K) $ is local,  and we replace  $S$ by $S_P$,  which is the integral closure of $R_P$.  Then
$S$ is semilocal,  and we denote the maximal ideals of $S$ by $\vect Q n$.  We want to show that $h$ is in 
$R + Q_1 \cap \cdots \cap Q_n$.  Let $c_i$ denote the image of $h$ in $L_i$,  $1 \leq i \leq n$,  where
$L_i = S/Q_i$.   Note that we may identify $K$ with a subfield of $L_i$ for every $i$. Note that $L_i$ is a finite
algebraic extension of  $K$ for every $i$.  We shall show that
all of the $c_i$ are in $K$,  and that they are all equal.  If $r \in R$ represents their common value,  then
$h-r$ is in all of the $Q_i$, which yields the desired conclusion. 

For each $Q_i$, choose a prime ideal $q_i$ of $S$ contained in $Q_i$ and maximal with respect to not containing
$f$.  Then $Q_i/q_i$ is a minimal prime of  $f(S/q_i)$,  and so $Q_i/q_i$  has height one.  Let $p_i = Q_i \cap R$.
Then $R/p_i$ is a local domain and since  $R/p_i \inj S/q_i$ is module-finite, we must have that $\di(R/p_i) = 1$,
by the dimension formula \cite[Theorem 15.6]{Mats}.  The image of $f$ is a non-zerodivisor in both $R/p_i$ and $(S/q_i)_{Q_i}$, and so 
is a non-zerodivisor in both of their completions.  We have an induced map of the completions  $C_i \to D_i$, which 
are one-dimensional reduced complete local rings.  Choose a minimal prime of $D_i$.  Its contraction to $C_i$ will 
not contain the image of $f$,  and so is also a minimal prime.   We get an induced map of quotient domains  
$\ov{C}_i \inj \ov{D}_i$. In each, the image of $f$ is nonzero.  Let $V_i$ be the normalization of $\ov{D}_i$; then 
$V_i$ is a complete local discrete valuation domain whose residue class field is an extension of $L_i$ and contains $K$.  

Let $W_i$ denote
the seminormalization of $\ov{C}_i$. Since  $V_i$ is normal, the normalization of $\ov{C}_i$  may be constructed
as a subring of  $V_i$,  and we may view $W_i \inc V_i$.   Then $W_i$ is a one-dimensional seminormal ring,
and we have a map $R \surj R/p_i \inj C_i \surj \ov{C}_i \inj W_i \inc V_i$.  Let $\cW_i$ denote the subring of 
$V_i$ consisting of all elements with image in $K$  modulo the maximal ideal of $V_i$.  Then  $\ov{C}_i \inc \cW_i$,
and whenever  $a \in V_i$ is such that $a^2, \, a^3 \in  \cW_i$,  one has that $a \in \cW_i$ as well.  Thus,  $W_i \inc \cW_i$.

We shall show  $c_i \in K$, using that the image  of $g$ is in  $fW_i$.  We have a commutative diagram
$$
\CD S @>{\beta}>> V_i\\
@AAA           @AAA\\
  R   @>{\alpha}>>    W_i\endCD 
$$
where the vertical maps are inclusions.  
Then  $\alpha(g) = \beta(g) = \beta(fh) = \beta(f)\beta(h) = \alpha(f)\beta(h)$.  Since $\alpha(g) \in \alpha(f)W_i$ (which holds because $g$ is in the axes closure of $fR$), and since $\alpha(f)$ is a non-zerodivisor on $V_i$, it follows that  $\beta(h) \in W_i$.  This implies
that  $\beta(h) \in \cW_i$,  and so its residue  $c_i$ is in $K$.  

Finally, suppose that we have $c_i$ and $c_j$ in $K$ for $i \not=j$.  We shall show $c_i = c_j$.  Let  $V_i$ and $V_j$
be as above.  Each is a complete discrete valuation domain whose residue class field is an algebraic extension
of $L_i$ and, hence, of $K$.  Enlarge $V_i$ to a complete discrete valuation domain $\cV_i$  whose residue class
field is the algebraic closure $\Omega$ of $K$.  Thus, there is a $K$-isomorphism $\theta$ between the residue class fields
of $\cV_i$ and $\cV_j$.  Let $F$ be the field in question.  Let $\cA = \cV_i \times_\theta \cV_j$ be the pullback of the surjection $\cV_i \times \cV_j  \twoheadrightarrow F \times F$ along the diagonal embedding $F \hookrightarrow F \times F$.  By Theorem~\ref{thm:glue}, $\cA$ is an excellent local
one-dimensional seminormal ring.  We have maps $\eta_i: S \to V_i \inc \cV_i$ and $(\eta_i,\,\eta_j)$ therefore maps
$S \to  \cV_i \times \cV_j$.  It is clear that the image of $R$ lies in $\cA$, since $\theta$ is a $K$-isomorphism.  Hence,
we have a commutative diagram:
$$
\CD  S @>{\beta}>>  \cV_i \times \cV_j \\
          @AAA                       @AAA \\
          R  @>{\alpha}>>       \cA \endCD .
 $$
where the vertical maps are inclusions.  Once again,  $\alpha(g) = \beta(g) = \beta(fh) = \beta(f)\beta(h) = \alpha(f)\beta(h)$,
and since $\alpha(g) \in \alpha(f)\cA$ and $\alpha(f)$ is a non-zerodivisor on $\cA$ (and hence also on $\cV_i \times \cV_j$, which is the normalization of $\cA$),  we must have $\beta(h) \in \cA$.
This implies that  the residues of $h$ correspond under $\theta$,  and since these are in $K$ and $\theta$ is a $K$-isomorphism,
they must be equal.  \end{proof}

\begin{cor}\label{cor:nzdcl}
Let $R$ be a reduced affine $\C$-algebra, and
let $f$ be a non-zerodivisor  of $R$.
Then $(fR)\ct = (fR)\ax = f S \cap R$, where $S$ is the seminormalization
of $R$. In particular, if $S$ is seminormal and $f$ is a non-zerodivisor,
$fS$ is both continuously closed and axes closed.
\end{cor}

\begin{proof}
 By Theorem~\ref{thm:semiaxes},  $fS$ is axes closed and, consequently, continuously closed as well.
Hence,  $(fR)\ax \inc (fS)\ax \cap R = fS \cap R$.  By Proposition~\ref{pr:ctcontract},   $(fR)\ct
= (fS)\ct \cap R$,  and since $fS$ is axes closed,  $(fS)\ct = fS$,  so that 
$(fR)\ct = fS \cap R$.  This shows that $(fR)\ax \inc (fR)\ct$,  and since we always
have the opposite inclusion, it follows that all three of $(fR)\ct, \, (fR)\ax$, and $fS \cap R$
are equal.
\end{proof}
 
 \section{Multiplying by invertible ideals and rings of dimension 2}\label{sec:dimtwo}
 In this section we prove that continuous closure and axes closure agree in locally factorial
 affine $\C$-algebras of dimension 2.  In particular, this holds for the polynomial ring in two
 variables over $\C$.  In \S\ref{sec:counter} we give an example which shows that they do not agree in the
 polynomial ring in three variables over $\C$.   In order to prove the main result of this section,
 we need some preliminary results.

 \begin{lemma}\label{lem:invert}  Let $^\#$ be a closure operation (see Definition~\ref{def:closure}) on $R$.  Suppose that for
 any two ideals $I,\, J$ or  $R$,   $I^{\#}J^{\#} \inc (IJ)^{\#}$.  Suppose further that 
 for any non-zerodivisor $r\in R$ and any ideal $J$ such that
 $J = J^{\#}$, we have $rJ = (rJ)^{\#}$.  Let $I$ be an ideal of $R$ that is locally free of rank one.
 Then for every ideal $J$ of $R$,   $I(J^{\#}) = (IJ)^{\#}$.
 In particular, if $I$ is locally free of rank one,  then $I = I^{\#}$.
  \end{lemma}  
 \begin{proof} Evidently, $I(J^{\#}) \inc I^{\#} J^{\#} \inc (IJ)^{\#}$, and so it suffices to show
 that  $(IJ)^{\#} \inc I(J^{\#})$.  Since the latter contains $IJ$,  it suffices to show that the
 latter is closed.  We may replace  $J$ by $J^{\#}$,  and so assume that $J = J^{\#}$,  and
 we want to prove that $IJ$  is closed.  Since $I$ is projective of rank one,  it is an invertible
 ideal, and we may choose an  ideal $I'$  such that  $I'I = rR$,   where $r$ is a non-zerodivisor.
 Then  $I'(IJ)^{\#} \inc (I'IJ)^{\#} = (rJ)^{\#} = r(J^{\#}) = (I'I)J^{\#} = I'\bigl(I(J^{\#})\bigr)$.  Multiplying
 by $I$ then yields that  $r\bigl((IJ)^{\#}\bigr) = r\bigl(I(J^{\#})\bigr)$.  Since $r$ is a non-zerodivisor,
 it follows that  $(IJ)^{\#} = I(J^{\#})$.  The final statement is the case where  $J = R$. \end{proof}

 \begin{thm}\label{thm:prinfaccont}  Let  $X$ be a closed affine algebraic set over $\C$,  and let
 $R = \C[X]$.  Let  $I$ be an ideal that is locally free of rank one,  and let $J$ be any ideal.   If  $R$ is seminormal,
 then $(IJ)\ct = I(J\ct)$.  \end{thm} 
 \begin{proof} By Lemma~\ref{lem:invert}, Proposition~\ref{pr:prodct}, and Corollary~\ref{cor:nzdcl}, 
 we may assume that $I = rR$,  where $r$ is a non-zerodivisor
 in $R$.   Also by Proposition~\ref{pr:prodct},   $r(J\ct) \inc (rJ)\ct$.  Now suppose
 that  $u \in (rJ)\ct \inc (rR)\ct$.  Since  $R$ is seminormal, we know from  Corollary~\ref{cor:nzdcl} that $rR$ is continuously
 closed,  and so we may write  $u = rf$ for some $f \in R$.  Then  $rf = u =  \sum_{i=1}^n g_irf_i$ (where $f_1, \dotsc, f_n$ is a generating set for the ideal $J$) for some $\vect g n \in \cC(X)$.  Since $r$ is not a zerodivisor in $\cC[X]$,  we have $f = \sum_{i=1}^n g_if_i$
 as well. \end{proof}
   
  \begin{thm}\label{thm:prinmax} Let $R$ be a seminormal excellent Noetherian ring.  Let $I$
 be an ideal of $R$ that is locally free of rank one and let $J$ be any ideal of $R$.  Then
 $(IJ)\ax = I(J\ax)$. \end{thm}
 \begin{proof} By Lemma~\ref{lem:invert}, Lemma~\ref{lem:prod}, and Theorem~\ref{thm:semiaxes}, we
 reduce at once to the case where $I = rR$ is generated by a non-zerodivisor $r$.  Since  $r(J\ax) \inc (rJ)\ax$,
 it suffices to prove the other conclusion,  which follows at once if $r(J\ax)$ is axes closed.  Therefore, we may
 replace $J$ by $J\ax$,  and it will suffice to show if $J$ is axes closed, then $rJ$ is axes closed.  Since
 $rR$ is axes closed by Theorem~\ref{thm:semiaxes},  it suffices to show that if $rf \in (rJ)\ax$  then
 $f \in J\ax$.  If we have a counterexample, by Theorem~\ref{thm:nzdtest}, there is a 
 map $h:R \to A$, where $A$ is  a one-dimensional excellent seminormal ring, $h(r)$ is a non-zerodivisor in $A$,  
 and $h(f)$ is not in $h(J)A$. Then $h(rf) = h(r)h(f) \notin  h(r)h(J)A = h(rJ)A$,  and so $rf \notin (rJ)\ax$, a contradiction.  
 \end{proof}
  
\begin{disc}\label{disc:factorial} Let $R$ be a locally factorial domain.  When $R$ is factorial, every nonzero
ideal $\fA$ is uniquely the product of a principal ideal (which may be $R$) and an ideal of height at least two (which
may also be $R$:  the height of the unit ideal is $+\infty$).  The principal ideal is generated by the greatest common
divisor of any given set of generators of $\fA$,  which is unique up to a unit multiplier,  and is the same as the greatest common
divisor of all elements of $\fA$.  When $R$ is only \emph{locally} factorial, we may say instead that every nonzero ideal
$\fA$ factors uniquely as the product of an ideal that is locally free of rank one and an ideal of height at least two.  
One can perform the factorization uniquely in every local ring of $R$, since the local rings are factorial.  But one can
actually carry this out on a cover by open affines:  it is clear that the factorizations on two affines will be the same on
the overlaps, since the factorization is unique in every local ring of $R$.  To get the factorization on a neighborhood
of a prime $Q$,  consider the height one primes $\vect P k$ of $R$ that contain $\fA$ and are contained in $Q$.  Each $P_i$
becomes principal when expanded to $R_Q$.  Localize $R$ at one element $f\notin Q$  so that each $P_iR_f$ is
principal, say  $\pi_iR_f$,  and so that the only height one primes of $R_f$ that contain $\fA R_f$ are the $P_iR_f$.  
Suppose that $\fA R_{P_i} = \pi_i^{a_i}R_{P_i}$.  Then $\fA R_f$   factors as  $rJ$  where  $r = \pi_1^{a_1}\cdots \pi_k^{a_k}$,
since comparing primary decompositions shows that $\fA \inc rR_f$.  The factor $J$ is not contained in any height
one prime of $R_f$,  and so this is the desired factorization.  

If, moreover, $R$ has dimension at most two,  then when we factor $\fA$ in this way,  the second factor $J$ is either the
unit ideal or is contained
only in maximal ideals of height two,  and is unmixed in the sense of having no embedded primes.  \end{disc}

\begin{thm}\label{thm:dimtwo} Let $R$ be a domain of dimension two that is a locally factorial affine $\C$-algebra.
Then axes closure and continuous closure agree for $R$. \end{thm}
\begin{proof} The result is immediate from the preceding discussion, Theorem~\ref{thm:prinfaccont},
Theorem~\ref{thm:prinmax}, and Corollary~\ref{cor:zerodimctax} \end{proof}

\section{A negative example and a fiber criterion for exclusion from the continuous closure}\label{sec:counter}

We begin with an inclusion lemma for axes closure.

\begin{lemma}\label{lem:colonax}
Let $R$ be an excellent Noetherian ring.  Let $I$ be an ideal and $\ici{I}$ the integral closure of $I$.  Let $P \in \Spec R$ and $J := (P\cdot \ici{I}) \cap (I :P)$.  Then $J \subseteq I\ax$.
\end{lemma}

\begin{proof}
Let  $f:R \to (A,\, \m)$ be a ring homomorphism, where $A$ is a complete local one-dimensional seminormal ring.   If $f(P) \nsubseteq \m$, then $JA \inc (I:P)A  = IA$.  Thus, we may assume that $f(P) \subseteq \m$.  In that case, the image of $J$ is contained in  $\m\icip{IA}$.  But by Lemma~\ref{lem:mJ}, $\m\icip{IA} \subseteq IA$.  Hence, $JA \subseteq IA$, as required.
\end{proof}

\begin{example}\label{ex:ctnotax}
In the polynomial ring $\C[u,v,x]$, the element $uvx$ is in the axes closure
of $I = (u^2, v^2, uvx^2)$ but not the continuous closure.

The first statement follows from Lemma~\ref{lem:colonax}, applied to the ideal $I$ and the prime ideal $P = (u,v,x)$, since $uv \in \ici{I}$.

Now suppose that $uvx$ is in the continuous closure of $I$,  say
\[
uvx = fu^2 + gv^2 + huvx^2,
\] where $f,g,h$ are continuous 
functions of $u,v,x$ in that order.   Let
$a = h(0,0,0)$.  Choose a constant $c \not=0$ such that $ch(0,0,c) \not = 1$:  
this is possible since  $xh(0,0,x) \to 0 \cdot a = 0$ as $x \to 0$.  
 Substitute $x = c$  in the displayed equation.
Then \[
cuv = f_0 u^2 + g_0 v^2 + h_0c^2 uv, 
\] where  the subscript indicates
the function of $u,v$ obtained by substituting $x = c$ in $f,g,h$ respectively.
The new equation involves only  $u,v$.  The function $c - h_0 c^2$ has value  $c\bigl(1 - c h(0,0,c)\bigr) \not= 0 $ at
the origin in the $u,v$-plane,  and so has a continuous inverse $s$ on a neighborhood $U$
of the origin.  Then  \[
uv = F u ^2 + G v^2, 
\] where the coefficients $F := s f_0$ and $G := s g_0$ are 
continuous functions defined on $U$.  
But this yields a contradiction.  To see this, let $A := \{(u,v) \mid u=v\} \cap U$ and $B := \{(u,v) \mid u=-v\} \cap U$.
On the set $A \setminus \z$, we have $F + G = 1$, so that by continuity, $F(0,0) + G(0,0) =1$.  But on $B \setminus \z$,
we have $F+G = -1$, so that $F(0,0) + G(0,0) = -1$ (again by continuity), whence $1=-1$, an absurdity.
\end{example}

\subsection*{Generalizing the counterexample}
If $R = \C[X]$ for an affine algebraic set $X$ and $x \in X$, recall from Discussion~\ref{disc:ctx} that we write $I\ctx$ for
the set of elements of the local ring $R_x$ of $\C[X]$ at $x$ that are continuous linear combinations
of elements of $I$ on some Euclidean neighborhood of $x$ in $X$.  This is the contraction to $R_x$
of the expansion of $I$ to the ring of germs of continuous (in the Euclidean topology) functions
on $X$ at $x$.

Suppose that $B \to R$ is a $\C$-homomorphism of finitely generated $\C$-algebras such 
that $B, R$ are reduced. Let $Y = \MaxSpec(B)$ and $X = \MaxSpec(R)$. 
Thus, we have a map $\pi:X \to Y$.  If $y \in Y$,  let  $\ov{R}^y$ denote the coordinate ring of the reduced fiber
 over $y$,  i.e., if $\m_y$ is the maximal ideal of $B$ corresponding to $y$,  then  $\ov{R}^y = (R/\m_yR)\red$.
$\MaxSpec(\ov{R}^y)$ may be identified with the fiber $X_y = \pi^{-1}(y)$,  and $\ov{R}^y$ with the ring of regular
functions on $\pi^{-1}(y)$.  We have a surjection $R \surj \ov{R}^y$ for every $y$, which may be
thought of as restriction of regular functions from $X$ to $X_y$.  If $g \in R$,  we write $g^y$
for the image of $g$ in $\ov{R}^y$.  

\begin{thm}[fiber criterion for exclusion from continuous closure]\label{thm:fibercrit}
 Let $B,\,R$ be as in the paragraph
above, and let notation be as in that paragraph.  Suppose that $f,\, g \in R$ and $I, \, J \inc R$ are
ideals. Suppose that:
\begin{enumerate}
\item $f \notin J\ct$ in $R$.
\item $\{ x \in X: g^{\pi(x)} \notin (I\ov{R}^{\pi(x)})\ctx\}$  is dense in $X$ in the Euclidean topology.
\end{enumerate}
Then $gf \notin (I+gJ)\ct$ in $R$. 
\end{thm}

Before giving the proof, we show how the example from the beginning of this section 
can be analyzed using this criterion.  Let $B= \C[x] \inc \C[x,u,v] = R$.  
Let $f = x$,  $g = uv$,  $I = (u^2,v^2)R$,  and $J = x^2R$.  Note that $f \notin J\ct$ in
$R$.  The fibers are simply the rings obtained by specializing $x$ to a complex constant,
and all of them may be identified  with $\C[u,v] \inc R$.  In this case,  $uv \notin (u^2,v^2)\ctx$
in all fibers $\C[u,v]$ for all $x$.  To see this, observe that $(u^2,\,v^2)$ is primary to the maximal ideal  $(u,v)$ 
of $\C[u,\,v]$.  It is clear that  $uv \notin (u^2,\,v^2)\ncl$, which is the same as $(u^2,\, v^2)\ct =
(u^2,\,v^2)\ax$ by Corollary~\ref{cor:zerodimctax}, and we may apply Proposition~\ref{pr:primmax}              
to conclude that $uv \notin (u^2,\,v^2)\ctz$.  Hence,  $xuv \notin (u^2,v^2, x^2uv)\ct$ in $R$. 

\begin{proof}[Proof of the fiber criterion]
Let $\vect u h$ generate $I$ and $\vect v k$ generate $J$.  Suppose that
\[
fg = \sum_{i=1}^h \alpha_i u_i + \sum_{j=1}^k \beta_j gv_j,
\]  where the the $\alpha_i$ and
$\beta_j$ are continuous.    Then  $\gamma = f - \sum_{j=1}^k \beta_j v_j$ is a continuous
function on $X$ that does not vanish identically, since $f \notin J\ct$.  Then $U = \gamma^{-1}(\C-\{0\})$
is a nonempty subset of $X$ that is open in the Euclidean topology.  Hence, it must meet 
$\{ x \in X: g^{\pi(x)} \notin (I\ov{R}^{\pi(x)})\ctx\}$.  Thus, we may choose $x \in X$ such that $\gamma(x) \not= 0$
(since $x \in U$) and $g^{\pi(x)} \notin (I\ov{R}^{\pi(x)})\ctx$.  But then there is a Euclidean neighborhood of $x$ on which
$\gamma$ does not vanish, so that $1/\gamma$ is a continuous function on this neighborhood, and, 
if $y = \pi(x)$,  on the intersection of this neighborhood with $X_y$ we have
\[
 g^y = \sum_{i = 1}^h \frac{\alpha_i(y)}{\gamma(y)} {u_i}^y,
\]
which shows that
$g \in (I\ov{R}^y)\ctx$,  a contradiction. 
\end{proof}

 \section{Mixed natural closure  and continuous closure for monomial ideals in polynomial rings}\label{sec:monomial}
 
 We  know that for an ideal $I$ of a finitely generated reduced $\C$-algebra, $I\ncl \inc I\ct  \inc I\ax$,  and that if
$I$ is unmixed, the conditions that it be naturally closed, continuously closed, and axes closed are all equivalent.
It turns out that when $I$ is mixed, even if it is a monomial ideal, a naturally closed ideal need not be
continuously closed and a continuously closed ideal need not be axes closed.  

Mixed natural closure is intended to make up for the following defect of  natural closure:  when $f \in R$,
it is possible that  $fg$ is not in the natural closure of  $fI$  even though $g$ is in the natural closure
of $I$.  (Note that it is obvious that if $g$ is in the continuous closure of $I$, then $fg$ is in the continuous
closure of $fI$.)

In fact, in the polynomial ring $\C[u,v,x]$ the ideal $(u^3x, v^3x)$ is naturally closed but not continuously
closed, since $u^2v^2x$ is in the continuous closure.  The point is that  $u^2v^2$ is in the natural closure
of $(u^3,v^3)$, and so it is in the continuous closure.  It follows that  $u^2v^2x$ is in the continuous closure
of $(u^3x, v^3x)$.  To see that $\mu = u^2v^2x$ is not in the natural closure, note that for $I$ monomial both
$I\ncl$ and   $I_{>1}$ are monomial.  Thus, it suffices to show that $\mu \notin I$ and $\mu \notin I_{>1} $.  The first
is clear.  The second is true since it is obviously true even after we localize at the prime ideal $(x)$.  

We note also that in $\C[u,v,x]$  the ideal  $(u^2, v^2, uvx^2)$ is continuously closed but not axes closed:  see
Example~\ref{ex:ctnotax} of the preceding section.

In this section, we first define a modified version of natural closure for ideals $I$ in any ring, which
we refer to as {\it mixed natural closure}.  We let $I\nmx$ denote the mixed natural closure of $I$. If $I = I\nmx$,
we say that $I$ is {\it \mnc}.  It will turn out that $I\ncl \inc I\nmx$, that a primary ideal is \mnc\ if and only if it is naturally 
closed. The definition will force it to be true that if  $fJ \inc I$ and $u \in J\nmx$,  then $fu \in I\nmx$.  

We shall eventually use the fiber criterion to prove that for monomial ideals in polynomial rings over $\C$, 
continuous closure always equals mixed natural closure.

\begin{defn}\label{defn:mnc}  We shall say that an ideal $I$ of a ring $R$ is {\it \mnc} if whenever $fJ \inc I$ and
$u \in J\ncl$,  then $fu \in I$.  An obviously equivalent condition is that for all $f \in R$,   $f\bigl((I:_Rf)\ncl\bigr) \inc I$.
Since $f$ may be taken to be 1,  an \mnc\ ideal is naturally closed.  
An intersection of \mnc\ ideals is clearly \mnc.  We define the {\it mixed natural closure} of $I$, denoted 
$I\nmx$, to be the intersection of all the MN-closed ideals containing  $I$, which is evidently the smallest
\mnc\ ideal containing $I$.  It is easy to see that $u \in I\nmx$ if and only if $(\dagger)$ there are finite sequences of ideals
$I = I_0, \, I_1, \, \ldots, \, I_h$,  $J_0,\, \ldots, J_{h-1}$ and finite sequences of elements  $\vect u h$ and $\vect f h$ in $R$
such that $u \in I_h$  and
such that for every $i$,  $0 \leq i < h$,  $f_{i+1}J_i \inc I_i$, $u_{i+1} \in J_i\ncl$,  and 
$I_{i+1} = I_i + f_{i+1}u_{i+1}R$. We note that we may take the ideals $J_i$ to be finitely generated,
and if $u_i \in J_i\ncl$ there is a subring  $R_0$ of $R$ finitely generated over $\Z$ and containing $u_i$, 
and an ideal $J_0$ of  $R_0$ generated by finitely many elements of $J_i$ such
that  $u_i \in J_0\ncl$.  It follows,
for example, that $I\nmx$ is the union of the ideals $(I \cap S)\nmx$ (in $S$) as $S$ runs through any directed
family of subrings of $R$ whose union is $R$.  In particular, we may let $S$ run through subrings of $R$
finitely generated over $\Z$. 

We give a third description of $I\nmx$.   Let $\fA_1(I) := \fA(I) := \sum_{f \in R}  f(I:_Rf)\ncl$.  Recursively, define
$\fA_{n+1}(I) = \fA\bigl(\fA_n(I)\bigr)$.  Note that $I \inc I\ncl \inc \fA(I)$ since we may take $f=1$.   
Then it is easy to see  $I = I\nmx$ if and only if $I = \fA(I)$ and that $I\nmx = \bigcup_{n=1}^\infty \fA_n(I)$.  
 \end{defn} \bigskip

 The following two Propositions are immediate from the definition and discussion above.

\begin{prop} Let $(R_\lambda, I_\lambda)$ be a directed family of rings and ideals (so that the map
$R_\lambda \to R_\mu$ takes $I_\lambda$ into $I_\mu)$.  Let $(R,\,I)$ be the direct limit, where $I \inc R$
is an ideal.  Then $I\nmx$ is the direct limit of the ideals $I_\lambda\nmx$ (calculated in the respective 
$R_\lambda$). \qed \end{prop}

\begin{prop} If $I \inc J$ are ideals of $R$, then $I\nmx \inc J \nmx$ and $(I\nmx)\nmx = I\nmx \supseteq I$.  In other words, mixed natural closure is a closure operation in the sense of Definition~\ref{def:closure}. \qed \end{prop}

We also have: 
\begin{prop}\label{pr:mnxinax} If $I \inc R$,  $I\ncl \inc I\nmx$. If $\blank^\#$ is any closure operation on ideals of $R$ such
that $I\ncl \inc I^\#$ for all ideals $I \inc R$ and $f\cdot(I^\#) \inc (fI)^\#$ for all $f \in R$ and all ideals $I \inc R$,  then 
$I\nmx \in I^\#$ for all $I$.  In particular, $I\nmx \inc \ov{I}$,  $I\nmx \inc I\ax$ whenever $I\ax$ is defined,  and $I\nmx \inc I\ct$
when $I\ct$ is defined.  Consequently, if $I = I\ct$,  then $I = I\nmx$.  Moreover,  if  $f$ is a non-zerodivisor in
a seminormal ring, then  $fR$ is \mnc. \end{prop}
\begin{proof} The statements are immediate except for the last statement.  This statement reduces to the
case where  $R$ is finitely generated over the integers and, hence, excellent.  The result now
follows from Theorem~\ref{thm:semiaxes}, since  $fR$ is axes closed.   \end{proof}
 
\begin{prop}\label{pr:mnclocal} If $R \to S$ is a homomorphism and $I$ is an ideal of $R$,   then $I\nmx$ 
maps into $(IS)\nmx$ in $S$. 
That is, mixed natural closure is persistent.  Hence, if $J$ is \mnc\ in $S$,  its contraction to  $R$ is \mnc\ in $R$.
If $W$ is any multiplicative system in $R$,  the mixed natural closure of $IW^{-1}S$ in $W^{-1}S$ is $I\nmx W^{-1}S$.
\end{prop}
\begin{proof}  The first statement follows at once from the recursive construction of the elements in the mixed natural
closure and the fact that natural closure is persistent, and the second statement is immediate from the first.
The fact that $I\nmx W^{-1}S \inc (IW^{-1}S)\nmx$ is a special case of persistence.  For the converse,
we use the final construction of $I\nmx$ in Definition~\ref{defn:mnc}. It suffices to show that $\fA(IW^{-1}S)
\inc W^{-1}\fA(IS)$. This follows at once because for all $f \in S$ and $w \in W$,
\[
IW^{-1}S:_{W^{-1}S} (f/w) = IW^{-1}S:_{W^{-1}S} (f/1) = (IS:_Sf)W^{-1}S;
\]
moreover, localization commutes with natural closure (see Proposition~\ref{pr:intloc}) and 
sum of ideals. \end{proof}

\begin{cor}  Let $I \inc R$ and $u \in R$.  Then $u \in I\nmx$ if and only if for every prime (respectively, maximal) ideal
$m$ of $R$,   $u/1 \in (IR_m)\nmx$. \qed \end{cor}

\begin{prop} If $I$ is primary to a prime ideal $P$ of $R$, then $I = I\nmx$ if and only if $I = I\ncl$.  \end{prop}
\begin{proof} It suffices to show that if $I = I\ncl$ then $I = I\nmx$.  For this we must show that if $fJ \inc I$ and $u \in J\ncl$,  
then $fu \in I$.  The condition that $u \in J\ncl$ can be expressed in terms of finitely many elements of $J$ and $R$ with
certain equations holding among them.  Therefore, we can choose a subring $R_0$ of $R$ that is finitely generated over
$\Z$ that contains $u$, $f$, and these elements.  Let $I_0$ be the contraction of $I$ to $R_0$,  $J_0$ the contraction
of $J$,  and let $P_0$ be the contraction of $P$.  Then it is still true that $I_0$ is primary to $P_0$,  that $I_0$ is naturally closed,
that $fJ_0 \inc I_0$,  and that $u$ is in the natural closure of $J_0$ in $R_0$.   It will suffice to show that $fu \in I_0$ in $R_0$.
Therefore, it suffices to consider the case where $R$ is a finitely generated algebra over $\Z$.  In particular, we may assume
without loss of generality that $R$ is excellent.

The result is now immediate from Theorem~\ref{thm:prnatax} and  Proposition~\ref{pr:mnxinax},   for we have
$I = I\ncl$ for a primary ideal implies $I = I\ax$ as well in the excellent case, and $I\nmx \inc I\ax = I$. \end{proof}

\begin{prop}  Let $J,\,I$ be ideals of the ring $R$ and let $f \in R$.  Suppose that $I = I\nmx$.  Then 
$fJ \inc I$ implies  $f\cdot(J\nmx) \inc I$. \end{prop}
\begin{proof}  If $u \in J\nmx$ then there are sequences of ideals and elements as in condition $(\dagger)$ of Definition~\ref{defn:mnc} 
such that $J$ is successively enlarged by one additional generator at a time until one reaches an ideal containing $u$.
It will suffice to show that each of the successive enlargements $J_i$ of $J$ has the property that $fJ_i \inc I$.  Since the base
step of the induction is simply that $fJ \inc I$, it will suffice to show that if $fJ_i \inc I$ and  $J_{i+1} = J_i + Rgv$,  where there is 
an ideal $\fA$ of $R$ and $g \in R$ such that $g\fA \inc J_i$ and $v \in \fA\ncl$,  then $fgv \in I$.   Since $fg\fA \inc
fJ_i \inc I$,  the result follows at once from the fact that $I = I\nmx$.  \end{proof}

\begin{rmk*} If an ideal $I$ is graded or multigraded, the same holds for its closure or another ideal derived from it in many instances.  
This is true for $I\ct$ by Proposition~\ref{pr:cthomog}, $I_{>1}$ and $I\ncl$ by Proposition~\ref{pr:nchomog}, as well as for
$I\ax$ by Proposition~\ref{pr:axgrad}.  In particular, in each of these cases,  if $I$ is a monomial ideal in a polynomial
ring over a field,  so is  $I\ct$, $I_{>1}$, $I\ncl$ and $I\ax$.  This is also true of $I\nmx$ if the ring contains an infinite
field.  We conjecture that this hypothesis is unnecessary. \end{rmk*}

\begin{prop}\label{pr:mncgrad} Let $R$ be a $\Z^h$-graded ring, $h>0$, and let $I$ be a homogeneous ideal with respect to this grading.
Assume that $R$ contains an infinite field.
Then $I\nmx$ is graded.  If $R$ is a polynomial ring over a ring $A$, then the MN-closure of every monomial ideal
is a monomial ideal.  In this later case, if $I$ is an ideal of $A$ that is \mnc,  then $IR$ is \mnc. \end{prop}
\begin{proof} All of the statements but the last are immediate from the Discussion~\ref{disc:grad}. 
For  the final statement, by a direct limit argument, we may assume that the number of variables is finite:  say
 $R = A[\vect x h]$. Let $F$ be a polynomial in the mixed natural closure of $IR$.
 It suffices to show that each term $a\mu$ of $F$,  where  $\mu$ is monomial in the $x_i$,  is such that $a \in I$.
 Since $(IR)\nmx$ is multigraded, we have that  $a\mu$ is in $(IR)\nmx$.  There is a unique $A$-homomorphism
 $R \to A$ that maps all the variables to 1.  Since mixed natural closure is persistent by Proposition~\ref{pr:mnclocal},  
 the image of $a\mu$,  which is $a$, is in the mixed natural closure of  $I$,  which is $I$.  \end{proof}

 \begin{rmk*}  By Discussion~\ref{disc:grad}, mixed natural closure preserves $\Z^h$-gradings whenever
 the ring contains arbitrarily large families of units such that the difference of any two distinct elements is also
 a  unit.  We conjecture that the results of Proposition~\ref{pr:mncgrad} hold without any hypothesis on $R$.
 However, we have not been able to prove in general that $(*)$ if $I$ is \mnc\ in  $A$,  then  $IA[x]$ is \mnc\ in the
 polynomial ring $A[x]$.  (Of course, if one knows this for one variable one gets the result for arbitrary sets
 of variables.)  If we simply knew this fact, it would imply that mixed natural closure commutes with
 $\Z^h$-gradings in general.  To see this, suppose that $R$ is $\Z^h$-multigraded and that $F \in I\nmx$
 for a graded ideal  $I$.  By Discussion~\ref{disc:grad}, if $s$ is a sufficiently large positive integer, 
 $\vect ts$ are new indeterminates of degree $(0,\,\ldots, \, 0)$,  and $g = \prod_{i=1}^s t_i \prod_{1 \leq j < k \leq s}(t_j-t_k)$,
 then each component $F_\lambda$ of $F$ is in the mixed natural closure of $IS$,  where  $S = R[\vect t s][1/g]$.  Since mixed
 natural closure commutes with localization by Proposition~\ref{pr:mnclocal}, there is an integer $N > 0$ such that $g^NF_\lambda$ 
 is in the mixed natural closure of   $IR[\vect ts]$.  Let $J$ be the mixed natural closure of $I$ in $R$.    If we know $(*)$,  we can 
 conclude that when $g^NF_\lambda$ is viewed as a polynomial in the $t_i$,  every coefficient is in $J$.  Hence, we have
 the desired result if some coefficient of $g^N F_\lambda$  is $F_\lambda$. It is therefore sufficient to note
 that for lexicographic order with $t_1 > \cdots > t_n$,  the highest order term in $g$  (and hence in $g^N$) has
 coefficient 1  (this term is $(\prod_{i=1}^s t_i)\prod_{j=1}^s t_j^{s-j})$.  \qed  \end{rmk*}


\begin{lemma}\label{lem:startmnb} Let $R$ be a polynomial ring over $\C$ in variables
$\vect x n$.   Let $I_1$ be a naturally closed ideal primary to the homogeneous maximal ideal $m_1$ in $R_1 = \C[\vect x k]$ and
suppose that \begin{equation}\tag{\#}\label{poundtag}
 \nu \text{ is a monomial in } I_1:_{R_1}m_1 \text{ such that } I_1 + \nu R_1 \text{ is naturally closed in } R_1.
\end{equation}
Let $J$ be a continuously closed monomial ideal of $R_2 = \C[x_{k+1}, \, \ldots, \, x_n]$.
Then $I_1R + \nu JR$ is continuously closed in $R$.  \end{lemma}

\begin{proof}   The continuous closure will be monomial.  Since $\nu$ is in the integral closure of $I_1$,
$I_1 + \nu R_1$ is naturally closed and primary in $R_1$ (or, if $\nu= 1$, it is all of $R_1$),  and so its expansion
to  $R$ is either $R$ or a primary naturally closed ideal.  Hence,
its expansion to $R$ is continuously closed.  It follows that the continuous closure of $I_1 + \nu J$ must have
the form $I_1 + \nu J'R$,  where $J' \supseteq J$ is a monomial ideal of $R_2$ (when $\nu$ is multiplied
by any variable in $R_1$, the product is in  $I_1R$ by hypothesis). 

Let $f$ be a monomial in $J' - J$ such that $f\nu \in (I_1 + JR)\ct$.  Let  $B = R_2$.  Let $g = \nu$. 
We apply the fiber criterion Theorem~\ref{thm:fibercrit} of the preceding section.  The fiber over a point
of $B$ is simply the result of specializing the $x_j$ for $j > k$ to complex constants given by the point.
Thus, all of the fibers may be identified with $R_1$,  and the image of $g = \nu$ is simply $\nu$. 
In this case,  $\nu \notin I_1\ctv$
in all fibers $R_1$ for all $v$ in $\MaxSpec(R)$.  To see this, observe that $I_1$ is primary to the maximal ideal  
$(\vect x k)R_1$ of $R_1$.  We have that $\nu \notin I_1 = I_1\ncl$, which is the same as $I_1\ct =
I_1\ax$ by Corollary~\ref{cor:zerodimctax}, and we may apply Proposition~\ref{pr:primmax}              
to conclude that $\nu \notin I_1\ctzz$.  But then $f \notin I_1R + \nu JR$ by  Theorem~\ref{thm:fibercrit}. \end{proof}

\begin{defn}\label{defn:mnb} Let $K$ be any field, and let $R = K[\vect x n]$ be a polynomial ring over $K$. 
We shall say that an ideal
$I$ of $R$ is {\it \mnb}  if there exist, for some positive integer $k$, mutually disjoint nonempty subsets $\vect S k$ of the variables,
monomial ideals $\vect I k$ such that for $1 \leq j \leq k$,  $I_j$ is naturally closed and primary to
the homogeneous maximal ideal $m_j$ of $R_j = K[x_t: x_t \in  S_j]$,  and for each $j$, where $1 \leq j \leq k-1$,
there is  a monomial 
$\mu_j$ such that $(\#_j)$ $\mu_j \in I_j:_{R_j}m_j$ and  $I_j + \mu_j R_j$ is naturally closed;  moreover,
$$
(\dagger) \quad I = (I_1 + \mu_1I_2 + \mu_1\mu_2I_3  + \cdots + \mu_1\mu_2\cdots \mu_{k-1}I_k)R.
$$
Note that if $k = 1$ this simply means that $I$ is a naturally closed primary monomial ideal.  An alternative recursive definition
is the following.  We allow among the \mnb\ ideals the primary naturally closed monomial ideals and those ideals obtained
by the following recursive rule.   If  $R_1$ and $R_2$ are polynomial subrings of $R$ generated by mutually disjoint subsets
of the variables, $I_1$ is a naturally closed monomial ideal primary to the homogeneous maximal ideal $m_1$ of $R_1$, 
 $\mu_1$ is a monomial of $R_1$ such that $(\#_1)$ $\mu_1 \in I_1:_{R_1}m_1$ and $I_1 + \mu_1 R_1$ is naturally closed, 
 and $J$ is an \mnb\ ideal in $R_2$,  then $(I_1 + \mu_1J)R$ is an \mnb\ ideal of $R$.
 
 We shall say that an \mnb\ ideal as above is {\it maximal} \mnb\ for the monomial $\mu$ if 
 $\vect S k$ is a partition of all of the variables $\{\vect x n\}$ and if there exists $\mu_k \in R_k$
 such $\mu = \mu_1 \, \cdots \, \mu_k$  and and for all $i$,  $1 \leq j \leq k$,  $I_j$ is maximal among
 naturally closed $m_j$-primary ideals in $R_j$ that do not contain $\mu_j$.  
\end{defn}

\begin{rmk*} Let $K$ be any field, let $R = K[\vect xn]$ be a polynomial ring, and let $R_1 = K[\vect x k]$
with homogeneous maximal ideal $m_1$.  Let $I_1$ be $m_1$-primary and naturally closed.  Suppose
that $\nu \in \ov{I_1} - I_1$.  Then it is automatic that condition (\ref{poundtag}) from Lemma~\ref{lem:startmnb} holds, i.e., that 
$\nu \in I_1:_{R_1} m_1$ and that $I_1' = I_1 + \nu R_1$ is
naturally closed.  The first part of the condition holds because $m_1 = \Rad(I_1)$, so that 
$m_1 \nu \inc \Rad(I_1) \ov{I_1} \inc (I_1)_{>1} \inc I_1$,
and the second because $(I_1')_{>1} \inc (\ov{I_1})_{>1} =(I_1)_{>1} \inc I_1 \inc I_1'$.   Also note that
in the definition $\nu$ can be 1,  but only if $I_1 = m_1$.  Finally, note that if $k = 1$,  the condition $(\#)$
holds precisely when $I_1 = x_1^{a+1}$ for $a \in \N$ and $\mu_1 = x_1^a$. \end{rmk*}

From the recursive version of \mnb\ monomial ideals and  Lemma~\ref{lem:startmnb}   
we have at once:

\begin{thm}\label{thm:mnbctclo} If $R$ is the polynomial ring  $\C[\vect x n]$,  every \mnb\ monomial ideal 
is continuously closed. \qed \end{thm}

\begin{cor} Let $K$ be a field of characteristic 0 and let $R$ be the polynomial ring $K[\vect x n]$.  
Then for every \mnb\ monomial ideal $I$,  $I\nmx = I$.  \end{cor}
\begin{proof}  $R$ is the directed union of the rings $K_0[\vect x n]$ where $K_0$ runs through all subfields
of $K$ finitely generated over $\Q$.  It therefore suffices to prove the result for a field $K$ that is finitely
generated over $\Q$.  Any such field $K$ is isomorphic to a subfield of $\C$, and so there is no loss of
generality in assuming that $K \inc \C$.  The given \mnb\ ideal $I$, expanded to $S = \C[\vect x n]$, is obviously still
\mnb.  Hence, $IS$ is continuously closed by the preceding result, and so $IS$ is \mnc\ by Proposition~\ref{pr:mnxinax}.  
Hence,  $IS \cap R = I$ is \mnc\ by Proposition~\ref{pr:mnclocal}.
\end{proof} 

\begin{rmk*} We conjecture that the result just above and the result just below hold for every field $K$,  
not only those of characteristic 0.
  \end{rmk*}
  
In Theorem~\ref{thm:intmnb} below we show that \mnc\ monomial ideals maximal with respect
to not containing a given monomial are maximal \mnb\ monomial ideals.  It will be convenient to have
the following result:

\begin{lemma}\label{lem:maxall} Let $K$ be a field of characteristic 0, let $R = K[\vect xn]$ be polynomial ring.  Let 
$\mu$ be a monomial.  Let $I$ be a monomial
ideal with $I\nmx = I$ such that $\mu \notin I$.  Let $S$ be the set of variables that do not occur in $\mu$,
and let $J = I + (x_j: x_j \in S)$.  Then $J\nmx = J$ and $\mu \notin J$.  Hence, a monomial ideal
$I$ which is maximal with respect to the conditions that $I\nmx = I$ and $\mu \notin I$ contains all
variables that do not occur in $\mu$. \end{lemma}

\begin{proof} After renumbering we may assume that $\vect x k$ occur in $\mu$, so that
$S = \{x_{k+1}, \, \ldots, \, x_n\}$.  Consider the $K$-algebra homomorphism $R \to R_1 = K[\vect x k]$
that fixes $x_j$ for $j \leq k$ and kills $x_j$ for $j \geq k+1$.  The image $I_1$ of $I$ in $R_1$ is the
monomial ideal spanned over $K$ precisely by those monomials in $I$ that are in $R_1$. 
We claim that $I_1\nmx = I_1$ in $R_1$,  for if $\nu \notin I_1$ were in the mixed natural closure
of $I_1$ in $R_1$,  then $\nu$ would be in $I\nmx = I$, since $I_1R \inc I$ (we may expand using
the inclusion $R_1 \inc R$).   This shows that $J\nmx = J$,  for $J$ is the contraction of $I_1$
under the surjection $R \to R_1$.  The final statement is now clear.  \end{proof}

\begin{thm}\label{thm:intmnb} Let $R = K[\vect xn]$ be a polynomial ring over a field $K$ of characteristic $0$.  
Then a monomial ideal $I \inc R$ is \mnc\ if and only if it is an intersection of \mnb\  monomial ideals.  Moreover, these
may be taken to be maximal \mnb\ monomial ideals.
\end{thm}
\begin{proof}   Let $\mu$ be a fixed monomial and let $I$ be an \mnc\ monomial
ideal that does not contain $\mu$.  Evidently, $I$ can be enlarged to a maximal \mnc\ monomial
ideal that does not contain $\mu$. Consequently, it will suffice to show that if $I$ is a maximal \mnc\ monomial
ideal for $\mu$, then it is maximal \mnb\ monomial ideal for $\mu$. We shall prove this by induction on $n$.   
If $\mu = 1$, then the only possible choice of $I$ is the homogeneous maximal
ideal, and this is maximal \mnb\ for 1.  Henceforth, we assume that $\mu \not= 1$.

We next handle the case where all of the $x_i$ occur in $\mu$.  

For each monomial $\alpha = x_1^{a_1} \cdots x_n^{a_n} \in K[ \vect x n]$,  let  $h(\alpha) = (\vect an) \in \N^n$. 
Consider the convex hull $C$ of points of $\N^n$ corresponding to  all monomials in $I$ together with
$h(\mu)$.     $h(\mu)$ must be a boundary point of $C$,  or else $\mu$ would be in $I_{>1}$.  
By a hyperplane in $\R^n$ we mean the translate of a vector subspace of dimension $n-1$.
We may choose a supporting hyperplane of $C$ that contains $h(\mu)$.   That is, 
there is a hyperplane through $h(\mu)$ such that $C$ lies entirely in one of the closed half-spaces
it determines. There is a nonzero real linear form  $L$ over $\R$ and $c \in \R$ such that this hyperplane
is defined by the equation $L = c$.  Let $\vect y n$ be real variables, and renumber the $y_i$ so that 
$\vect  y k$ are the real variables that have nonzero coefficients in $L$: at least one coefficient
is nonzero.   

Thus, we may assume that
the equation of the hyperplane is $m_1 y_1 + \cdots m_k y_k = c$, where the $m_i \in \R -\{0\}$.  
 
By multiplying by $-1$ if necessary, we may assume 
that $c \geq 0$.   We may assume that all the coefficients $m_i$ are positive.  (To see this, note that
all points with sufficiently large coordinates represent elements in $I$ and  will lie on one side of this 
hyperplane.  If $m_i$ is positive (respectively, negative), choose a large value $N \in \N$ for the $y_j$,
$j \not = i$,   and a very large positive integer value $B$ for the value of $m_i$.  
 The value of $L = m_1y_1 + \cdots + m_ky_k$ will be $> c$ (respectively, $< c$) for $B \gg 0$.  
 Thus, if there are coefficients 
 with different signs, not all points with large coordinates are on the same side of the hyperplane.)  
 Thus, we may assume that all the $m_i$ have the same sign.  By multiplying by $-1$ we may assume
 that all of them are positive.  Note that $L$ evaluated at $h(\mu)$ or $h(\nu_1)$ is $c$  (the value of $L$ only depends on
the first $k$ entries of the vector).  Since $h(\mu)$ has no nonzero entries, $c >0$.  
Write $\mu = \mu_1 \theta$,  where  $\mu_1$ involves $\vect xk$ and $\theta$ involves the other variables.
 \medskip

Let $I_1 \subseteq K[\vect xk] = R_1$ be generated by all monomials $\lambda$ except $\mu_1$ in $\vect xk$ 
such that the value of $L$ at $h(\lambda)$ is $\geq c$.  Note that $I_1$ is primary to $(\vect x k)$:  
since each coefficient of $L$ is positive,
the functional will be $> c$ on $Ne_i$,  $1 \leq i  \leq k$, when $N \gg 0$. \medskip

Moreover,  $I_1 + \mu_1 R_1$ is integrally closed:  the set of exponent vectors contains all lattice
points in its convex hull because it is the intersection of a half-space, the first orthant, and the set
of lattice points.   In particular, it is naturally closed.   It is also the case that $I_1$ contains
$x_i \mu_1$ for $1 \leq i \leq k$:   the value of $L$ is evidently $\geq c$ for each of these, and
each of these is different from $\mu_1$.  We also claim that $I_1$ is naturally closed.  It suffices to show that $\mu_1$ is not in the natural closure, and for this it suffices to
see that $\mu_1 \notin (I_1)_{>1}$.  But if $\mu_1^t \in I_1^{t+1}$ then $t h(\mu_1)$ is the
sum of $t+1$ values of  $h$ on points of $I_1$.  When we apply  $L$,  we find that  $tL\bigl(h(\mu_1)\bigr)$
is the sum of $t+1$ real numbers, each of which is $\geq c$.  This implies that $tc \geq (t+1)c$,  a
contradiction.  

Moreover, $I_1R + \mu_1 R$ is the ideal of $R$ generated by all monomials in $R$ on
which the value of $L$ is $\geq c$.  Clearly,   $I \inc I_1R + \nu R$:  given any monomial in $I$,  the
value of $L$ on  the associated vector is $\geq c$,  and this value depends only on that part of
the monomial involving $\vect x k$:  the latter must be in $(I_1,\, \nu) K[\vect x k]$.  The monomials in $I$
that are not in $I_1R$ must be monomials of the form $\alpha \beta$  where $\alpha \in K[\vect x k]$
and $\beta \in K[x_{k+1} , \, \ldots, \, x_n]$.  Since they are in $I$,  the value of the functional on
$h(\alpha)$ must be $\geq c$  which means that $\alpha \in I$ unless $\alpha = \nu$.  Thus, $I \inc
I_1R+  J_0\nu R$,  where $J_0 \inc K[x_{k+1}, \, \ldots, \, x_n]$ contains those monomials whose product
with $\nu$ is in $I$.  $\theta$ cannot be in the mixed natural closure of $J$,  or else $\mu = \nu \theta$
will be in the mixed natural closure of $\nu JR \inc I$.   Thus,  we may enlarge
$J_0$ to a maximal mixed naturally closed ideal $J$ of $K[x_{k+1},\, \dotsc,\, x_n]$ maximal with respect to not 
containing $\theta$.   

By the induction hypothesis we have that $J$ is a maximal \mnb\ monomial ideal of $K[x_{k+1},\, \dotsc,\, x_n]$
maximal with respect to not containing $\theta$.  It follows that $I_1R + JR$ is a maximal \mnb\ ideal
for $\mu = \mu_1 \theta$.  Since $I$ is contained in this ideal, $I$ must be equal to $I_1 + JR$. 

This completes the treatment of the case where all of the $x_i$ occur in $\mu$.  In the general case,
by renumbering, suppose that $\vect x h$ are the variables that do not occur in $\mu$.  Giving a maximal
monomial \mnc\ ideal $I$ such that $\mu \notin I$ is equivalent to give such an ideal $\fA$  for $\mu$ in 
$K[x_{h+1}, \, \ldots, \, x_n]$ and then enlarging it to include $\vect x h$,  by Lemma~\ref{lem:maxall}.
The enlarged ideal will still be maximal \mnb:  if $I_1$ is the initial primary ideal in the representation
$(\dagger)$  of $\fA$ as in Definition~\ref{defn:mnb},  one can simply enlarge $I_1$ to contain $\vect x h$.  
\end{proof}

We now come to the motivating result for this section:  

\begin{thm} For a monomial ideal $I$ in the polynomial ring $\C[\vect xn]$,  $I\ct = I\nmx$. \end{thm}
\begin{proof} We observed earlier that $I\nmx \inc I\ct$.  Thus, it will suffice to show that $I\nmx$ is
continuously closed.  But by Theorem~\ref{thm:intmnb}, it is an intersection of \mnb\ ideals, and by 
Theorem~\ref{thm:mnbctclo} every \mnb\ ideal is continuously closed. \end{proof}

\begin{example} Note that it is not true that continuous closure agrees with mixed natural closure in reduced
affine algebras over $\C$.  In $R = \C[x^2,\,x^3, xy,\,y]$, a subring of the polynomial ring  $S= \C[x,\,y]$,  we have that  the principal ideal $I = yR$ is such that $I = I\nmx$. (We know that $I\nmx$ is monomial and is contained in $(y,xy)R$, since this 
is the contraction of  $y\C[x,\,y]$.  Hence, it will suffice to show that $xy$ is not in  $f(I:f)\ncl$ for any choice of  $f$.
Since $xy$ is irreducible in  $R$,  the only possibilities for $f$ are, up to multiplication by nonzero elements
of $\C$, 1  and  $xy$.   The choice $f = 1$ does not work since $I\ncl = I$,  and the choice
$f = xy$ does not work since  $I:_Rf$ is proper monomial, and so its natural closure is contained in the 
homogeneous maximal ideal $m$ of  $R$, while  $xy \notin gym$.)  However,  $xy \in I\ct$.  \end{example}

\section{A bigger axes closure}\label{sec:bigax}
In deciding on a generalization of Brenner's notion of axes closure to arbitrary Noetherian rings, we had a choice between whether we would base it on seminormal rings or so-called \emph{weakly normal} rings.

\begin{defn}\cite{AnBom-wn}
Let $R$ be a reduced Noetherian ring, and $\icr{R}$ the integral closure of $R$ in its fraction field.  Then the \emph{weak normalization} $R^{\mathrm{wn}}$ of $R$ is the set of all elements $x\in \icr{R}$ that satisfy the following property for all $\p \in \Spec R$:

If $R/\p$ has prime characteristic $p>0$, let $\pi(\p) :=p$; otherwise let $\pi(\p) :=1$.  Then there is some positive integer $n$ such that $(x/1)^{\pi(\p)^n} \in R_\p + \Jac(\icr{R}_\p)$.    

We say that $R$ is {\it weakly normal} if  $R = R^{\mathrm{wn}}$.
\end{defn}

It is clear that if $R$ is weakly normal, it is also seminormal.  Moreover, if $R$ has equal characteristic $0$, the weak normalization is of course the same as the seminormalization of $R$, and in particular in the finitely generated $\C$-algebra case, they agree.  
Therefore, with the following definition, one has that $I\ax = I\AX$ for any ideal $I$ in a $\C$-algebra.

\begin{defn}
Let $R$ be a Noetherian ring, $I$ an ideal of $R$, and $f\in R$.  We write $f\in I\AX$ if for every map from $R$ to an excellent one-dimensional weakly normal ring $S$, the image of $f$ is in $IS$.
\end{defn}

It is quite straightforward to verify that $I \mapsto I\AX$ is a closure operation. 

\begin{prop} For every ideal $I \inc R$, a Noetherian ring,  $I\ax \subseteq I\AX \subseteq \ici{I}$.\end{prop}
\begin{proof} This is clear, since normal $\implies$ seminormal $\implies$ weakly normal.  \end{proof}

One has the following parallel to Proposition~\ref{pr:semibasic}.  Items (6) and (7) follow in this case for the same 
reasons their analogues did in the seminormal case:

\begin{prop}\label{pr:weakbasic}
Suppose $R$, $S$ are reduced Noetherian rings.  Let $\icr{R}$ be the integral closure of $R$ in its total quotient ring. \begin{enumerate}
\item Suppose $R$ is seminormal.  Then $R$ is weakly normal if and only if for any prime integer $p$ and any $x \in \icr{R}$ such that $x^p, px \in R$, we have $x\in R$. \cite[Proposition 1]{Ito-wn}
\item If $g: R \to S$ is faithfully flat and $S$ is weakly normal, then $R$ is weakly normal. \cite[Corollary II.2]{Man-wn}
\item If $R$ is weakly normal and $W$ is a multiplicative set, then $W^{-1}R$ is weakly normal. \cite[Corollary IV.2]{Man-wn}
\item Suppose the integral closure of $R$ in its total quotient ring is module-finite over $R$.  The following are equivalent: \cite[Corollary IV.4]{Man-wn} \begin{enumerate}
 \item $R$ is weakly normal.
 \item $R_\m$ is weakly normal for all $\m \in \MaxSpec R$.
 \item $R_\p$ is weakly normal for all $\p \in \Spec R$.
 \item $R_\p$ is weakly normal for all $\p\in \Spec R$ such that $\depth R_\p=1$.
 \end{enumerate}
\item Suppose $g: R \to S$ is flat with geometrically reduced (e.g. normal) fibers and $\icr{R}$ is module-finite over $R$.  If $R$ is weakly normal, then so is $S$. \cite[Proposition III.3]{Man-wn}
In particular, if $S$ is smooth over $R$, which includes the case where $S$ is \'etale over $R$, and $R$ is weakly normal, then $S$
is weakly normal.
\item A directed union of weakly normal rings is weakly normal.  
\item If $R$ is local and weakly normal, then the Henselization of $R$ and the strict Henselization of $R$ are weakly normal.
\item Suppose $R$ is excellent and local.  $R$ is weakly normal $\iff \hat{R}$ is weakly normal. \cite[Proposition III.5]{Man-wn}
\item Let $X$ be an indeterminate over $R$.  $R$ is weakly normal $\iff R[\![X]\!]$ is weakly normal. \cite[Proposition III.7]{Man-wn}
\end{enumerate}
\end{prop}

For reasons parallel to observations in the $\ax$ case, it suffices to consider only maps to where $S$ is local, or even complete local.  Many properties which hold for $I\ax$ have analogies in $I\AX$.  To see how this works, we offer the following analogue to Theorem~\ref{thm:glue}.

\begin{thm}\label{thm:weakglue}
Let $k$ be a field  of prime characteristic $p>0$, let $L_1, \dotsc, L_n$ be finite algebraic extension fields of $k$ such that under the diagonal embedding $k \rightarrow L_1 \times \cdots \times L_n$, the image of $k$ is $p$th-root closed.  Let $(V_i, \m_i)$ be discrete valuation rings such that $V_i / \m_i \cong L_i$.  Let $S$ be the subring of 
$\prod_{i=1}^n V_i$ consisting of all $n$-tuples $(v_1, \dotsc, v_n)$ such that there exists $\alpha \in k$ such that $v_i \equiv \alpha\ (\md \m_i)$ for all $i$.  Then $S$ is weakly normal.

Conversely, let $(R,\m, k)$ be a \emph{complete} one-dimensional weakly normal Noetherian local ring, where $k$ has prime characteristic $p>0$.  Then there exist such extension fields $L_i$ and such DVRs $V_i$ (which moreover are complete) such that $R$ is isomorphic to the ring $S$ described above.
\end{thm}

\begin{proof}
If $R$ is weakly normal, then it is seminormal, so it has the form given in Theorem~\ref{thm:glue}.  One must only check that $k$ is $p\,$th-root closed in $L_1 \times \cdots \times L_n$.  So let $p$ be the characteristic of $k$, which must therefore agree with the characteristics of all the $L_i$.  We may assume $p>0$.  Let $c = (c_1, \dotsc, c_n) \in \prod_{i=1}^n L_i$ such that $c^p \in k$.  Let $v = (v_1, \dotsc, v_n) \in \prod V_i$ such that $c_i$ is the residue class of $v_i$ mod $\m_i$, for each $i$.  Then $\ov{pv} = p\ov{v} = 0 \in k$, so that $pv \in R$, and $\ov{v^p} = c^p \in k$, so that $v^p \in R$.  Since $R$ is weakly normal, it follows that $v\in R$.

Conversely, suppose $R$ is constructed in the way outlined in the statement of the theorem.  Say $p$ is the characteristic of $k$.  Without loss of generality, $p>0$.  Let $q$ be a prime integer and let $v \in \prod_{i=1}^n V_i$ be such that $v^q, qv \in R$.  If $q\neq p$, then since $q$ is a unit in all of the $L_i$, it follows that it is invertible in $\prod L_i$.  So $q \ov{v} \in k$ implies that $\ov{v} \in k$, which then implies that $v\in R$ by the description of $R$.  On the other hand, if 
$v = (v_1, \dotsc, v_n) \in \prod V_i$ is such that $v^p, pv \in R$, then, in particular, $\ov{v}^p \in k$, and since $k$ is $p\,$th-root closed in $\prod L_i$, it follows that $\ov{v} \in k$, so that $v\in R$.
\end{proof}

It follows, for example, that a local or complete ring of axes over any field is weakly normal, since the diagonal 
embedding in question is the usual one, $k \rightarrow k \times \cdots \times k$, in which it is clear that the image of $k$ is $p$th-root closed.

We also have the following parallel to Proposition~\ref{pr:onedimsemi}
\begin{prop}\label{pr:onedimweak} Let  $L$ be an algebraically closed field. 
\begin{enumerate}
\item[(a)]  A complete axes ring over $L$ is weakly normal.
\item[(b)] Every complete local one-dimensional
weakly normal ring of equal characteristic with algebraically closed residue
class field $L$ is isomorphic with a complete ring of axes  over $L$.  
\item[(c)] Every affine ring of axes over $L$ is weakly normal.
\item[(d)] A one-dimensional affine $L$-algebra $R$ is weakly normal if and only if there
are finitely many \' etale $L$-algebra maps $\theta_i : R \to A_i$,  where the $A_i$ are
affine rings of axes over $L$, and every maximal ideal of $R$ lies under a maximal ideal
of some $A_i$.
\end{enumerate}
\end{prop}

Combining this with Proposition~\ref{pr:onedimsemi}, it follows that for complete one-dimensional rings with algebraically closed residue field, and for finitely generated one-dimensional algebras over an algebraically closed field, there is no difference between weak normality and seminormality.

\begin{proof}
For part (a), we use the characterization in Theorem~\ref{thm:weakglue}, noting that one obtains no extra $p$th roots from the diagonal embedding $L \to L \times \cdots \times L$.  For part (b), any complete local one-dimensional weakly normal ring of equal characteristic with residue field $L$ is in particular seminormal, so by Proposition~\ref{pr:onedimsemi}(b), it is isomorphic to a complete ring of axes over $L$.  Part (c) then follows immediately from part (a) and from parts 4b and 8 of Proposition~\ref{pr:weakbasic}.

The proof of part (d) follows from the argument in the proof of Proposition~\ref{pr:onedimsemi}(d), using the corresponding parts of Proposition~\ref{pr:weakbasic} in place of where we had previously used Proposition~\ref{pr:semibasic}.
\end{proof}

However, $\AX$ is really too big for our purposes here.  Consider the following:

\begin{example}
Let $k$ be a field of characteristic $p>0$, let $t$, $x$ be analytic indeterminates, and let $R = k(t^p)[\![x,tx]\!]$.  Then we claim that $R$ is seminormal, but also that $ \icr{R} = R^{\mathrm{wn}} = k(t)[\![x]\!]$.  The statement about the weak normalization follows from the fact that $t^p, pt=0 \in R$.  The resulting ring is obviously normal, so $\icr{R} = R^{\mathrm{wn}} = k(t)[\![x]\!]$. To see that $R$ is seminormal, take any $f \in \icr{R}$ such that $f^2, f^3 \in R$.  Then if $f_0$ is the constant term of the power series $f$, we have $f_0^2, f_0^3 \in k(t^p)$, whence $f_0 = f_0^3/f_0^2 \in k(t^p)$.  But $R$ is exactly the set of all $f\in \icr{R}$ whose constant term is in $k(t^p)$.

Now let $\m = (x, tx)R$ be the unique maximal ideal of $R$, and let $I = xR$. Note that $I$ is $\m$-primary. Then $I\AX = x R^{\mathrm{wn}} \cap R = (x,tx)R = \m$, but $I\ncl = I\ax=I$ because $R$ is a one-dimensional complete seminormal ring.  Hence, \begin{enumerate}
\item $\ax$ and $\AX$ do not always agree, even for $\m$-primary ideals in 1-dimensional complete local domains, and
\item $\ncl$ and $\AX$ do not always agree, even for $\m$-primary ideals in 1-dimensional complete local domains.
\end{enumerate}
\end{example}

Property (1) is perhaps not surprising, but property (2) means that the closure is too big to apply our methods \emph{mutatis mutandis}:
since one lacks the property that $I\AX=I\ncl$ for primary ideals, it is not clear how one would prove an analogue of Theorem~\ref{thm:nzdtest} or of the crucial Theorem~\ref{thm:semiaxes} in this new context (substituting $\AX$ for $\ax$ and weakly normal for seminormal everywhere).  Thus, this bigger axes closure does not appear to be as suitable for our main purpose here as the smaller one.  However, we have provided some of the fundamentals in this section because it may be useful in other situations.

\providecommand{\bysame}{\leavevmode\hbox to3em{\hrulefill}\thinspace}
\providecommand{\MR}{\relax\ifhmode\unskip\space\fi MR }
\providecommand{\MRhref}[2]{%
  \href{http://www.ams.org/mathscinet-getitem?mr=#1}{#2}
}
\providecommand{\href}[2]{#2}

\end{document}